\documentclass[11pt]{amsart}

\catcode`\@=11

\usepackage{amsmath,amsfonts,amscd,amssymb}

\theoremstyle{plain}
\def\theequation{\thesection.\arabic{equation}}
\newtheorem{theorem}[equation]{Theorem}
\newtheorem*{theorem*}{Theorem}
\newtheorem{problem}[equation]{Problem}
\newtheorem{assumption}[equation]{Assumption}
\newtheorem{conjecture}[equation]{Conjecture}
\newtheorem{proposition}[equation]{Proposition}
\newtheorem{lemma}[equation]{Lemma}
\newtheorem{corollary}[equation]{Corollary}
\theoremstyle{definition}

\newtheorem{remark}[equation]{Remark}
\newtheorem*{claim}{Claim}

\newtheorem{definition}[equation]{Definition}
\newtheorem*{strategy}{Strategy}
\newtheorem*{definition*}{Definition}
\newtheorem{notation}[equation]{Notation}
\newtheorem{example}[equation]{Example}

\newtheoremstyle{namedthm}{}{}{\slshape}{}{\bfseries}{. }{0em}%
  {\thmnote{#3}\thmnumber{ #2}}
\theoremstyle{namedthm}

\newtheorem*{named*}{}

\newcommand{\vabove}[2]{\genfrac{}{}{0pt}{}{#1}{#2}}

\def\smallmatrix#1#2#3#4{\genfrac{(}{.}{0pt}{1}{#1}{#3}\genfrac{.}{)}{0pt}{1}{#2}{#4}}

\def\leftchoice#1#2#3#4{{\def\arraystretch{0.7}
\Bigl\{\!\!\begin{array}{ll}
   \scriptstyle #1,\!\!\!&\scriptstyle #2\cr
   \scriptstyle #3,\!\!\!&\scriptstyle #4\end{array}}}
\def\bigleftchoice#1#2#3#4{{\def\arraystretch{1}
   \biggl\{\!\!\begin{array}{ll}#1,&#2\cr#3,&#4\end{array}}}

\def\rksel#1#2#3{\rk_{#3}(#1/#2)}

\def\lara{\langle,\rangle}
\def\blangle{\boldsymbol{\langle}}
\def\brangle{\boldsymbol{\rangle}}

\let\iff\Leftrightarrow
\let\liff\Longleftrightarrow
\def\injects{\lhook\joinrel\rightarrow}

\let\surjects\twoheadrightarrow

\let\lar\longrightarrow
\let\iso\cong
\let\tensor\otimes
\let\normal\triangleleft

\def\inputdocument#1{\begingroup
  \def\documentclass{%
  \let\oldinput=\input\let\input=\skiptoken
  \let\oldusepackage=\input\let\usepackage=\skiptoken
  \futurelet\comingchar\andocstyle}%
  \def\andocstyle{\ifx\comingchar[\let\skipdoc=\skipdocstyle\else\let
  \skipdoc=\skiptoken\fi\skipdoc}%
  \def\document{\let\input=\oldinput\let\usepackage=\oldusepackage}%
  \def\enddocument{}%
  \input{#1}\endgroup}
\def\skiptoken#1{}
\def\skipdocstyle[#1]#2{}

\def\beq{$$\begin{array}{llllllllllllllll}}
\def\eeq{\end{array}$$}
\def\beqn{\begin{equation}\begin{array}{llllllllllllllll}}
\def\eeqn{\end{array}\end{equation}}
\def\beql#1{\begin{equation}\label{#1}\displaystyle}
\def\eeql{\end{equation}}

\input cyracc.def       
\font\tencyr=wncyr10
\def\sha{\text{\tencyr Sh}}

\def\ufootnote#1{\insert\footins{\noindent\footnotesize{{\hskip
  1.5em}\llap{${}^{\vphantom a}$}#1}}}

\def\tbuildrel#1\over#2{\buildrel\text{\rm\normalsize\smaller[3]#1}\over{#2}}

\def\thincdots{\raise1.3pt\hbox{$\scriptscriptstyle
  \>\cdot\>\cdot\>\cdot\>\cdot\hskip0.3pt$}}

\def\<{\raise0.5pt\hbox{$\,\scriptstyle<\,$}}

\def\arrowlim#1#2{\mathop{\underset{\scriptstyle #1}{\underset
    {\raisebox{0ex}[0.25ex][-0.5ex]{$#2$}}{\operatorname{lim}}}}}
\newcommand{\invlim}[1][]{\arrowlim{#1}{\longleftarrow}}        

\def\bb@symb#1|#2{\expandafter\def\csname #2#1\endcsname{\mathbb{#1}}}
\def\bbsymbols#1#2{\@for\@tmpz:=#2\do{\expandafter\bb@symb\@tmpz|{#1}}}

\def\cal@symb#1|#2{\expandafter\def\csname #2#1\endcsname{\mathcal{#1}}}
\def\calsymbols#1#2{\@for\@tmpz:=#2\do{\expandafter\cal@symb\@tmpz|{#1}}}

\def\dmth@p#1|{\expandafter\let\csname#1\endcsname\relax
  \expandafter\DeclareMathOperator\csname#1\endcsname{#1}}
\def\operators#1{\@for\@tmpz:=#1\do{\expandafter\dmth@p\@tmpz|}}

\operators{Tr,rk,ord,Gal,BSD,Hom,tors,coker,Ind,Res,Frob,GL,PGL,SL,Aut}
\operators{Sel,End,Maps,Reg,Sy,Syl,sgn,id}

\let\Re\vRe

\def\Ql{\Q_l}
\def\Zl{\Z_l}
\def\Qp{\Q_p}
\def\Zp{\Z_p}

\bbsymbols{}{R,C,Z,N,Q,F}
\calsymbols{}{O}

\bbsymbols{}{A,P,G}
\bbsymbols{b}{H}
\calsymbols{}{E,L,T,B,X}
\calsymbols{c}{C,A,P,O,K,G,N,F,E,H,Q,R,M,L}
\operators{Sel,Spec,Norm}

\def\p{{\mathfrak p}}

\def\triv{{\mathbf 1}}

\def\kbar{\bar k}
\def\Kpbar{\overline{K_\p}}

\def\Kp{K_\p}
\def\GKp{G_{\Kp}}
\def\IKp{I_{\Kp}}

\def\rkalg#1#2{\rk #1/#2}
\def\rksel#1#2#3{\rk_{#3}#1/#2}

\def\daggerequation#1#2#3{\begingroup      
  \def\theequation{#2}
  \begin{equation}
  \label{#1}
    #3
  \end{equation}\endgroup
  \addtocounter{equation}{-1}}
\def\CC#1#2{{C_{\hbox{\tiny$\scriptscriptstyle #1/#2$}}}}

\def\smminusone{(\hskip -0.15mm\raise0.243em\hbox to 5pt{\hrulefill}\hskip 0.15mm1\hskip -0.15mm)}
\def\Esmslash{\hskip -0.15mmE\hskip -0.3mm/\hskip -0.3mm}

\def\smallcoprod{\raise1.3pt\hbox{$\,\scriptstyle\coprod\,$}}

\@addtoreset{equation}{section}

\begin{document}
\let\oldthepage\thepage\def\thepage{}

\title{Notes on the Parity Conjecture}
\author{Tim Dokchitser}
\address{Robinson College, Cambridge CB3 9AN, United Kingdom}
\date{September 19, 2010}
\email{t.dokchitser@dpmms.cam.ac.uk}
\maketitle
\ufootnote{The author is supported by a Royal Society University Research Fellowship}

\vfill

The main purpose of these notes is to prove, in a reasonably self-contained
way, that finiteness of the Tate-Shafarevich group
implies the parity conjecture for
elliptic curves over number fields.
Along the way, we review local and global root numbers of elliptic
curves and their classification, and we end by discussing
some peculiar consequences of the parity conjecture.

Essentially nothing here is new,
and the notes follow closely the papers
\cite{Isogroot}--\cite{Antifor},
all joint with Vladimir Dokchitser.
There are only some additional shortcuts
replacing a few technical computations of root numbers and
Tamagawa numbers by a `deforming to totally real fields' argument.
Also, this is not meant to be a complete survey of results on the
parity conjecture, and many important results, especially those
concerning Selmer groups, are not mentioned.

The exposition is based on the lectures given at CRM Barcelona in
December 2009.
It is a great pleasure to thank CRM for its warm hospitality and
Victor Rotger and Luis Dieulefait for organising the lecture series.
I~would also like to thank its participants, especially
Anton Mellit, Ravi Ramakrishna and Thomas de la Rochefoucauld
for many useful comments, Alex Bartel for his helpful lecture series
preceding this one, and Vladimir Dokchitser for proofreading the
manuscript.

\vfill
\vfill

\newpage

\llap{.\hskip 10cm}
\vfill
\tableofcontents
\vfill

\newpage

\let\thepage\oldthepage

\section{Birch-Swinnerton-Dyer and Parity}

\subsection{Conjectures and the main result}
\label{ssmainres}

Throughout the notes $K$ denotes a number field.
Suppose $E/K$ is an elliptic curve, say in Weierstrass form,
$$
  y^2=x^3+ax+b, \qquad a,b\in K.
$$
The set $E(K)$ of $K$-rational solutions $(x,y)$ to this equation together
with a point at infinity $O$ forms an abelian group, which is finitely
generated by famous theorems of Mordell and Weil.
The primary arithmetic invariant of $E/K$ is its rank:

\begin{definition*}
The {\em Mordell-Weil rank\/} $\rk E/K$ is the $\Z$-rank of $E(K)/$torsion.
\end{definition*}

The celebrated conjecture of Birch and Swinnerton-Dyer relates it to
another fundamental invariant, the {\em $L$-function\/} $L(E/K,s)$.
We define it in \S\ref{s:root},
together with its {\em conductor} $N\in\Z$ and the
{\em global root number} $w(E/K)=\pm 1$.
For now it suffices to say that
$L(E/K,s)$ is given by a Dirichlet series,
$$
  L(E/K,s) = \sum_{n=1}^\infty a_n n^{-s}
$$
with $a_n\in\Z$, and it converges for $\Re s>3/2$.

\begin{named*}[Hasse-Weil Conjecture]
The function $L(E/K,s)$ has an analytic continuation to the whole of $\C$.
The completed $L$-function
$$
  L^*(E/K,s) = \Gamma(\tfrac s2)^d\Gamma(\tfrac{s+1}2)^d
    \pi^{-ds} N^{s/2} L(E/K,s)
    \qquad\qquad(d=[K:\Q])
$$
satisfies a functional equation
$$
  L^*(E/K,2-s) = w(E/K) L^*(E/K,s).
$$
\end{named*}

Granting the analytic continuation, we may state

\begin{named*}[Birch-Swinnerton-Dyer Conjecture I]
$$
  \ord_{s=1}L(E/K,s) = \rk E/K.
$$
\end{named*}

This remarkable conjecture relates arithmetic properties of $E$ to
analytic properties of its $L$-function. It was
originally stated over the rationals, extended
to all abelian varieties over global fields by Tate and vastly
generalised by Deligne, Gross, Beilinson, Bloch and Kato.
One immediate consequence of the two conjectures above is

\begin{named*}[Parity Conjecture]
$(-1)^{\rk E/K}=w(E/K)$.
\end{named*}

In particular, an elliptic curve whose root number is $-1$ must have
infinitely many rational points. This is a purely arithmetic statement
which does not involve $L$-functions and it might appear to be simpler
than the two conjectures above. However, it is remarkably hard, and
we have no approach to resolve it in any kind of generality.
It has several important consequences: for instance,
it settles Hilbert's 10th problem over rings of integers of
arbitrary number fields \cite{MR3} and implies most remaining cases of
the congruent number problem.  

The difficulty is that we know virtually nothing about the rank of a general
elliptic curve, as it is very hard to distinguish rational points from the
elements of $\sha$:

\begin{definition*}
The {\em Tate-Shafarevich group} of $E/K$ is defined by
$$
  \sha_{E/K} = \ker \bigl(
    H^1(K,E(\bar K))\!\to\!\prod_v H^1(K_v,E(\bar K_v))
  \bigr).
$$
\end{definition*}

The famous Shafarevich-Tate conjecture asserts that $\sha_{E/K}$ is finite.
The main result result of these notes is that its finiteness does imply
the parity conjecture. (Over function fields, Artin and Tate \cite{TatC}
showed that the full Birch--Swinnerton-Dyer conjecture follows from
finiteness of $\sha$.)

We write $E[n]$ for the $n$-torsion in $E(\bar K)$, and $K(E[n])$ for the
field obtained by adjoining to $K$ the coordinates of points in $E[n]$.

\let\oldtheequation\theequation\def\theequation{A}
\begin{theorem}
\label{imain}
Let $E$ be an elliptic curve over a number field~$K$, and suppose
that $\sha_{E/K(E[2])}$ has finite 2- and 3-primary parts. Then
$$
  (-1)^{\rk E/K} = w(E/K).
$$
\end{theorem}
\let\theequation\oldtheequation
\setcounter{equation}{0}

The problem that finiteness of $\sha$ implies parity
has a reasonably long history.
It was solved for elliptic curves over $\Q$ with $j$-invariant 0 or 1728
by Birch and Stephens \cite{BS}, for CM elliptic curves over $\Q$
by Greenberg \cite{Gre} and Guo \cite{Guo}, all (modular)
elliptic curves over $\Q$ by Monsky \cite{Mon} and most modular
elliptic curves
over totally real fields by Nekov\'a\v r \cite{NekS}.
Over arbitrary number fields the theorem above is proved in \cite{Squarity}
for elliptic curves with `decent' reduction types at 2 and 3, and for
all elliptic curves in~\cite{Kurast}.

%
%

\subsection{Birch--Swinnerton-Dyer II and isogeny invariance}

To explain our approach to the parity conjecture, we need to state
the second part of the Birch--Swinnerton-Dyer conjecture.


\def\neron#1{\omega_{#1}^o}

\begin{notation}
\label{mainnot}
Denote the discriminant of $K$ by $\Delta_K$.
For an elliptic curve $E/K$ write
$E(K)_{\tors}$ for the {\em torsion\/} subgroup of $E(K)$
and
$$
  \Reg_{E/K}=\det(\langle P_i,P_j\rangle)
$$
for the {\em regulator\/}
of $E/K$; here $\{P_i\}$ is any basis of $E(K)/E(K)_{\tors}$ and
$\lara$ the N\'eron-Tate height pairing over $K$.

Finally, we define the product $C_{E/K}$ of {\em Tamagawa numbers}
and {\em periods\/}.
Fix an invariant differential $\omega$ on $E$, and
let $\neron{v}$ denote the N\'eron differential
at a finite place $v$ of~$K$. Set
$$
  C_{E/K} = \prod_{v\nmid\infty} c_v \left|\frac{\omega}{\neron{v}}\right|_{_v}
    \>\cdot\>\prod_{\vabove{v|\infty}{\text{real}}} \int\limits_{E(K_v)}\!\! |\omega|
    \>\cdot\>\prod_{\vabove{v|\infty}{\text{cplx}}} 2\!\!\!\int\limits_{E(K_v)}\!\! \omega\wedge \bar\omega,
$$
with $c_v$ the local Tamagawa number at $v$ and 
$|\cdot|_v$ the nomalised absolute value on $K_v$.
We sometimes denote the individual terms of $C_{E/K}$
by $C(E/K_v,\omega)$, so that
$C_{E/K}=\prod_v C(E/K_v,\omega)$. The terms depend on the choice of $\omega$
but their product does not: replacing $\omega$ by $\alpha\omega$ with
$\alpha\in K^\times$ changes it by $\prod_v |\alpha|_v$, which is 1
by the product formula.
\end{notation}

\begin{named*}[Birch--Swinnerton-Dyer Conjecture II]
The Tate-Shafarevich group $\sha_{E/K}$ is finite,
and the leading coefficient of $L(E/K,s)$ at $s=1$ is
\begingroup
$$
  \frac{|\sha_{E/K}|\Reg_{E/K}\,C_{E/K}}{|E(K)_{\tors}|^2\sqrt{|\Delta_K|}}
  \quad\qquad=:\BSD_{E/K}.
$$
\endgroup
\end{named*}

Over general number fields, the only thing known about $\BSD_{E/K}$
is that it is an isogeny invariant, which is a theorem of Cassels \cite{CasIV}.
If $\phi: E\to E'$ is an isogeny over $K$, then $L(E/K,s)=L(E'/K,s)$
($\phi$ induces an isomorphism
$V_l(E)\to V_l(E')$, and the $L$-function is defined in terms of $V_l$),
so in effect this means that BSD II is compatible with isogenies.
Actually, because isogenous curves have the same rank
($\phi$ induces $E(K)\tensor\Q\iso E'(K)\tensor\Q$), so is BSD I as well.
We will need a slight generalisation of Cassels' theorem,
which relies on isogeny
invariance of BSD II for general abelian varieties:

\begin{lemma}
\label{lrel}
Suppose $K_i$ are number fields, $E_i/K_i$ are elliptic curves
and $n_i$ are integers.
If $\prod_i L(E_i/K_i,s)^{n_i}=1$ and all $\sha_{E_i/K_i}$ are finite, then
$\prod_i \BSD_{E_i/K_i}{}^{n_i}=1$.
\end{lemma}

\begin{proof}
Taking the terms with $n_i<0$ to the other side (and renaming
the curves and fields if necessary),
rewrite the assumed identity of $L$-functions
$$
  \prod_i L(E_i/K_i,s)=\prod_j L(E'_j/K'_j,s).
$$
Let $A=\prod_i\Res_{K_i/\Q}E_i$ and $A'=\prod_j\Res_{K'_j/\Q}E_j$
be the products of Weil restrictions of the curves to $\Q$.
These are abelian varieties with
$L(A/\Q,s)=L(A'/\Q,s)$, so $A$ and $A'$ have isomorphic
$l$-adic representations (Serre \cite{SerA} \S2.5 Rmk. (3)), and are
therefore isogenous (Faltings \cite{Fa}). Their
$\sha$'s are assumed to be finite and BSD-quotients are invariant
under Weil restriction (Milne \cite{MilO} \S1)
and isogeny (Tate--Milne \cite{MilA} Thm. 7.3). This proves the claim.
\end{proof}

\subsection{Parity example}
\label{s91b1}

Why is this relevant to the parity conjecture, which concerns
only the first part of BSD?
Assume for the moment finiteness of $\sha$, and consider the following
example. The two elliptic curves over $\Q$
\beq
  E:  & y^2+y=x^3+x^2-7x+5   &&&\Delta_E    = -7\cdot 13&&(91b1)\cr
  E': & y^2+y=x^3+x^2+13x+42 &&&\Delta_{E'} = -7^3\cdot 13^3&&(91b2)\cr
\eeq
are isogenous via a 3-isogeny $\phi: E\to E'$ defined over $\Q$.
Choose $\omega,\omega'$ ($=\tfrac{dx}{2y+1}$) to be the global minimal
differentials, so $C_p=c_p$ for all $p$. The curves $E, E'$ have split
multiplicative reduction at 7 at 13, and their local Tamagawa numbers
and the infinite periods $c_\infty=C(E/\R,\omega)$, $c'_\infty=C(E'/\R,\omega')$ are
\beq
  c_7 = v_7(\Delta_E) =1, & c_{13}=v_{13}(\Delta_E)=1; & c_\infty=6.039... \cr
  c'_7 = v_7(\Delta_{E'}) =3, & c'_{13}=v_{13}(\Delta_{E'})=1; & c'_\infty=2.013...
\eeq
In fact, $c_\infty=3c'_\infty$ (see the computation below for general
isogenous curves).
Now $\BSD_{E/\Q}=\BSD_{E'/\Q}$ by Cassels' theorem, but
$$
  \left.
  \begin{array}{lll}
  C_{E/\Q}&=&1\cdot 3\cdot c_\infty = c_\infty \cr
  C_{E'/\Q}&=&3\cdot 3\cdot \tfrac13 c_\infty = 3c_\infty
  \end{array}
  \quad\right]\>\>\text{not equal,}
$$
so some other terms in the Birch--Swinnerton-Dyer constant for $E$ and $E'$
must be unequal as well. Because the two are off by a rational non-square
factor, and the conjectural orders of $\sha$ (as well as of (torsion)$^2$)
are squares, the regulators must be unequal! In other words,
$$
  \frac{\Reg_{E/\Q}}{\Reg_{E'/\Q}} = \frac{C_{E'/\Q}}{C_{E/\Q}} \cdot
    \frac{|\sha_{E'/\Q}|}{|\sha_{E/\Q}|} \cdot
    \frac{|E(\Q)_{\tors}|}{|E'(\Q)_{\tors}|} = 3 \cdot \square\cdot\square \ne 1.
$$
If now $\rk E/\Q$ were 0, then so would be $\rk E'/\Q$ and we would have
$$
  \Reg_{E/\Q}=\Reg_{E'/\Q}=1,
$$
contradicting the above. So, assuming finiteness of $\sha$, we proved that
{\em $E/\Q$ has positive rank}. In fact,
{\em $E/\Q$ has odd rank}:

\begin{lemma}
\label{lregisog}
Let $\phi: E/K\to E'/K$ be a $K$-rational isogeny of degree $d$.
Then
$$
  \frac{\Reg_{E/K}}{\Reg_{E'/K}} = d^{\rk{E/K}} \cdot \text{(rational square)}.
$$
\end{lemma}

\begin{proof}
Write $n=\rk E/K=\rk E'/K$, and pick a basis
$$
  \Lambda:=E(K)/\text{torsion} = \langle P_1,...,P_n\rangle, \quad
  \Lambda':=E'(K)/\text{torsion}.
$$
Write $\phi^t: E'\to E$ for the dual isogeny, so that $\phi^t\phi=[d]$. Then
\beq
  d^n \Reg_{E/K} &= \det\langle dP_i,P_j\rangle = \det\langle \phi^t\phi P_i,P_j\rangle\cr
    &= \det\langle \phi P_i,\phi P_j\rangle
    = \Reg_{E'/K} \cdot [\Lambda':\phi(\Lambda)]^2.
\eeq
\end{proof}

Returning to our example, we proved that
$$
  3\cdot\square = \frac{\Reg_{E/\Q}}{\Reg_{E'/\Q}} = 3^{\rk E/\Q} \cdot \square,
$$
so the rank is odd. Because $w(E/\Q)=(-1)^3=-1$ (2 split places + 1 infinite;
see \S\ref{s:root}), we showed that
$$
  \text{Finiteness of $\sha$} \quad\implies\quad \text{Parity Conjecture}
$$
for the curve $E=91b1/\Q$.

This is the main idea behind the proof
of Theorem \ref{imain} in general. We will use equalities of $L$-functions
to deduce relations between rank parities and the $C$'s, and then verify
by means of local computations that the latter agree with the required root
numbers.

\section{The $p$-parity conjecture}

To prove Theorem \ref{imain},
along the way we will also need slightly finer statements that give
unconditional results for the Selmer ranks of $E/K$.

Recall that $\sha$ is an abelian torsion group, and
for every prime $p$ its $p$-primary part can be written as
$$
  \sha_{E/K}[p^\infty] \iso (\Q_p/\Z_p)^{\delta_p}\times \text{(finite $p$-group of square order)}.
$$

\begin{definition*}
With $\delta_p$ as above,
the {\em $p$-infinity Selmer rank\/} of $E/K$ is
$$
  \rk_p E/K = \rk E/K + \delta_p.
$$
\end{definition*}

If $\sha$ is finite as expected, then $\delta_p=0$ and $\rk_p=\rk$ for all $p$.
But the point is that we
can often say something about $\rk_p$ without assuming finiteness of $\sha$.
In particular, the following version of the parity conjecture is
often accessible:

\begin{named*}[$p$-Parity Conjecture]
$(-1)^{\rk_p E/K}=w(E/K)$.
\end{named*}

This conjecture is known over $\Q$ (\cite{Mon,NekS,Kim,Squarity})
and for `most' elliptic curves over
totally real fields \cite{NekIV,Kurast,NekSC};
see Theorem \ref{2parthm} below.

\begin{remark}
\label{remrkp}
It is more conventional to define $\rk_p$ as follows. Let
$$
  X_p(E/K) = \Hom_{\scriptscriptstyle{\Z_p}}\bigl(\varinjlim_n\Sel_{p^n}(E/K), \Q_p/\Z_p\bigr)
$$
be the Pontryagin dual of the $p^\infty$-Selmer group of $E/K$.
This is a finitely generated $\Z_p$-module, and
$$
  \X_p(E/K) = X_p(E/K)\tensor_{\Zp}\Qp
$$
is a $\Q_p$-vector space, whose dimension is precisely $\rk_p$.
Both $X_p$ and $\X_p$ are (contravariantly) functorial in $E$ and they
behave well under field extensions. Specifically, if $F/K$ is a finite Galois extension,
its Galois group acts on $\X_p(E/F)$, and there is a canonical isomorphism
of the Galois invariants with $\X_p(E/K)$,
\beql{selinv}
  \X_p(E/K) = \X_p(E/F)^{\Gal(F/K)}.
\eeql
(The proof is a simple inflation-restriction argument, see e.g.
\cite{Squarity} Lemma 4.14.). For $p\nmid [F:K]$ this is even true on
the level of $X_p$. Because the same Galois invariance holds for
$E(K)\tensor_\Z\Q_p\subset \X_p(E/K)$, we get the following corollary
(which can also be proved directly).
\end{remark}

\begin{corollary}
\label{finsha}
Suppose $F/K$ is a finite extension of number fields,
and $E/K$ is an elliptic curve. If $\sha_{E/F}$ is finite,
then so is $\sha_{E/K}$.
\end{corollary}

\begin{remark}
The definitions above and Corollary \ref{finsha} apply also to
abelian varieties in place of elliptic curves.
\end{remark}

\subsection{Proof of Theorem \ref{imain}}

To establish Theorem \ref{imain}, we will need to prove the
$2$-parity conjecture for elliptic curves with a $K$-rational
2-torsion point and a special case of the $3$-parity conjecture:

\begin{theorem}[=Theorem \ref{2isogthm}]
\label{i2isogthm}
Let $K$ be a number field, and $E/K$ an elliptic curve with
a $K$-rational 2-torsion point $O\ne P\in E(K)[2]$.
Then
$$
  (-1)^{\rk_2 E/K}=(-1)^{\ord_2\frac{C_{E/K}}{C_{E'/K}}}=w(E/K),
$$
where $E'=E/\{O,P\}$ is the 2-isogenous curve.
\end{theorem}

\begin{theorem}
\label{ithms3glo}
Let $F/K$ be an $S_3$-extension of number fields, and let
$M$ and $L$ be intermediate fields of degree 2 and 3 over $K$, respectively.
For every elliptic curve $E$ over $K$,
$$
  \vphantom{\int^\int}
  \smminusone^{\hbox{\smaller[4]$\rksel EK3\!+\!\rksel EM3\!+\!\rksel EL3$}} \!=\!
    \smash{\smminusone^{\ord_3\!\frac{\CC EF\CC EK^2}{\CC EM\CC EL^2}}}
  \!=\! w(\Esmslash K)w(\Esmslash M)w(\Esmslash L).
$$
\end{theorem}

These two parity statements are sufficient to deduce Theorem \ref{imain}:

\begin{proof}[Proof of Theorem \ref{imain}]
Take $F=K(E[2])$, so $\Gal(F/K)\subset\GL_2(\F_2)\iso S_3$.
By Corollary \ref{finsha},
the 2- and 3-primary parts of $\sha_{E/k}$ are finite
for $K\!\subset\! k\!\subset\! F$.

If $E$ has a $K$-rational 2-torsion point, the result follows from
Theorem~\ref{i2isogthm}. If $F/K$ is cubic, then $\rkalg EK$ and
$\rkalg EF$ have the same parity, because $E(F)\tensor\Q$ is a rational
$\Gal(F/K)\iso C_3$-representation, so its dimension has the same parity
as that of its $C_3$-invariants. Also, $w(E/K)=w(E/F)$ (see \cite{KT} p.167,
or prove this directly using the results of \S\ref{s:root}),
so the result again follows.

We are left with the case when
$\Gal(F/K)\iso S_3$.
Let $M$ be the quadratic extension of $K$ in $F$
and $L$ one of the cubic ones. By the above argument,
\beq
  \rkalg EM \text{ even} & \Longleftrightarrow & w(E/M)=1,\cr
  \rkalg EL \text{ even} & \Longleftrightarrow & w(E/L)=1.\cr
\eeq
On the other hand, by Theorem \ref{ithms3glo},
$$
    \rkalg EK\!+\!\rkalg EM\!+\!\rkalg EL\text{ is even}\>\>\Longleftrightarrow\>\>
    w(E/K)w(E/M)w(E/L)=1,
$$
and Theorem \ref{imain} is proved.
\end{proof}

\subsection{Local formulae for the Selmer parity}

There are two totally different steps in proving
Theorems \ref{i2isogthm} and \ref{ithms3glo}, one global and one local.
The first one is to relate the Selmer parity to the products of periods
and Tamagawa numbers (the first equality in the two theorems),
and the second one is a local computation comparing the latter to
root numbers (the second equality).

The first one is in the spirit of the example in \S\ref{s91b1}.
Recall that there we had two elliptic curves over $\Q$ and an equality of
$L$-functions,
$$
  L(E/\Q,s)=L(E'/\Q,s),
$$
that originated from an isogeny $E\to E'$, say of degree $d$.
Assuming finiteness of $\sha$, it implies an equality of the
Birch--Swinnerton-Dyer quotients $\BSD_{E/\Q}=\BSD_{E'/\Q}$,
which reads modulo squares
$$
  \frac{\Reg_{E'/\Q}}{\Reg_{E/\Q}} \equiv
  \frac{C_{E/\Q}}{C_{E'/\Q}} \mod \Q^{\times 2}.
$$
The `lattice index' argument (Lemma \ref{lregisog}) shows that the
left-hand side is the same as $d^{\rk E/\Q}$ modulo squares. So
we have an expression for the parity of $\rk E/\Q$,
which is a difficult {\em global} invariant in terms of easy
{\em local} data.
This is an absolutely crucial step in all parity-related proofs.
In can be made slightly finer, without assuming finiteness of $\sha$,
but at the expense of working with Selmer groups.
The precise statement is as follows:

\begin{notation}
For an isogeny $\phi: A\to A'$ of abelian varieties over $K$
we write $\phi^t: (A')^t\to A^t$ for the dual isogeny.
For a prime $p$ write
$$
  \phi_p: \X_p(A'/K)\to \X_p(A/K)
$$
for the induced map on the dual Selmer groups, and
$$
  \phi_{p,v}: A(K_v)\to A'(K_v)
$$ 
for the map on local points.
We let $\chi(\cdot):=|\ker(\cdot)|/|\coker(\cdot)|$.
\end{notation}

\begin{theorem}
\label{thmQ}
Let $\phi: A\to A'$ be a non-zero isogeny
of abelian varieties defined over a number field $K$.
For every prime $p$,
$$
  \frac{\chi(\phi_p^t)}{\chi(\phi_p)}
    =\text{ $p$-part of } \prod_v \chi(\phi_{p,v})
    =\text{ $p$-part of } \frac{C_{A/K} }{C_{A'/K}}.
$$
\end{theorem}

The proof is an application of Poutou-Tate duality
(see e.g. \cite{Squarity} \S4.1).
For example, if $\phi: E\to E'$ is an isogeny of elliptic curves of
prime degree~$p$, then
$$
  \frac{C_{E/K}}{C_{E'/K}}
     =
  \frac{\chi(\phi_p^t)}{\chi(\phi_p)}
     \equiv
  \chi(\phi^t_p)\chi_p(\phi_p)
     = \chi([p]_p) = p^{\rksel EKp} \pmod{\Q^{\times2}},
$$
which is a formula of Cassels (see Birch \cite{Bir} or Fisher \cite{FisA}).
This extends the argument of \S\ref{s91b1}
and in particular Lemma \ref{lregisog} to an unconditional
statement for Selmer groups, and proves the first equality
of Theorem~\ref{i2isogthm} (with $p=2$).

\subsection{Parity in $S_3$-extensions}
\label{s:s3}

We now proved the first `global' step of theorem \ref{i2isogthm}.
Now we do the same for Theorem \ref{ithms3glo}, in other words
we claim the following\footnote{
  This is \cite{Squarity} Thm.~4.11 with $p=3$. Note that
  the contributions from $v|\infty$ to
  ${\frac{\CC EF\CC EK^2}{\CC EM\CC EL^2}}$
  cancel when using the same $K$-rational $\omega$ over each field.
  The definition of $C$ in \cite{Squarity} excludes infinite places,
  so the
  formula there does not need the ${\frac{\CC EK^2}{\CC EL^2}}$ term,
  as it is then a rational square.}

\begin{theorem}
\label{s3rksel}
Let
$F/K$ be an $S_3$-extension of number fields, $M$ and $L$~
inter\-mediate fields of degree 2 and 3 over $K$, and $E/K$ an
elliptic curve. Then
$$
  \rksel EK3+\rksel EM3+\rksel EL3 \equiv \ord_3\tfrac{\CC EF\CC EK^2}{\CC EM\CC EL^2} \mod 2.
$$
\end{theorem}

As it was done in \S\ref{s91b1} for the isogeny case, we
start with a slightly simpler version first, assuming finiteness of $\sha$.
Again, we will use a relation between $L$-functions, but this time
it is one of a different nature.

Thus, suppose $G=\Gal(F/K)\iso S_3$ and $M$ and $L$ are
as in the theorem. Let $E/K$ be an elliptic curve for which
$\sha_{E/F}$ is finite. By Corollary \ref{finsha},
$\sha_{E/K}$, $\sha_{E/M}$ and $\sha_{E/L}$ are finite as well.

The group $S_3$ has 3 irreducible representations, namely
$\triv{}$ (trivial), $\epsilon$ (sign) and a 2-dimensional representation
$\rho$, all defined over $\Q$.
The list of subgroups of $S_3$ up to conjugacy is
$$
  \cH=\{1,C_2,C_3,S_3\},
$$
and they correspond by Galois theory to $F, L, M$ and $K$ respectively.
Each $H\in\cH$ gives rise to a representation $\C[G/H]$ of $G$,
associated to the $G$-action on the left cosets of $H$ in $G$.
Because there are {\em four} subgroups and only {\em three} irreducible
representations, there is a relation between these. Writing out
$$
  \C[G]\iso \triv\oplus\epsilon\oplus\rho\oplus\rho,\quad
  \C[G/C_3]\iso \triv\oplus\epsilon\quad\text{and}\quad
  \C[G/C_2]\iso \triv\oplus\rho,
$$
we find that the (unique such) relation is
\beql{eqs3rel}
  \C[S_3] \oplus
  \C[S_3/S_3]^{\oplus 2}
     \iso
  \C[S_3/C_3] \oplus \C[S_3/C_2]^{\oplus 2}.
\eeql
Now tensor this relation with the $l$-adic representation
$V_l(E/K)_\C=V_l(E/K)\tensor_{\Q_l}\C$ (embedding $\Q_l\injects\C$ in some way).
By the Artin formalism for $L$-functions (see \S\ref{s:root}),
for every $H\in\cH$,
$$
  L(V_l(E/K)_\C\tensor\C[G/H],s) = L(E/F^H,s),
$$
so we have a relation between $L$-functions\footnote
{The elliptic curve has nothing to do with this: this relation already
exists on the level of Dededind zeta-functions of the number fields.
This forces relations between the regulators and class groups
of number fields, and they have been studied by Brauer, Kuroda, de Smit
and others (see e.g. \cite{Bar} for references).},
$$
  L(E/F,s)L(E/K,s)^2 = L(E/M,s)L(E/L,s)^2.
$$
Applying Lemma \ref{lrel}, we get a relation between the regulators
and the products of periods and Tamagawa numbers,
\beql{rc1}
  \frac{\Reg_{E/F}\Reg_{E/K}^2}{\Reg_{E/M}\Reg_{E/L}^2} \equiv
  \frac{C_{E/M}C_{E/L}^{2}}{C_{E/F}C_{E/K}^{2}} \pmod{\Q^{\times2}}\>.
\eeql
How do we interpret the left-hand side, and why is it even a rational number?
Tensor the Mordell-Weil group $E(F)$ with $\Q$ and decompose it as a
$G$-representation,
$$
  V = E(F)\tensor\Q \iso \triv^{\oplus a}\oplus \epsilon^{\oplus b}
                     \oplus \rho^{\oplus c}.
$$
Next, compute the ranks of $E$ over the intermediate fields of $F/K$
in terms of $a$, $b$ and $c$: using Frobenius reciprocity,
for a subgroup $H$ of $S_3$ we have
$$
  \rk E/F^H \!=\! \dim V^H \!=\! \langle \triv, V|_H\rangle_H
            \!=\! \langle \C[G/H], V\rangle_G \!=\!
  \Biggl\{
  \begingroup
  \smaller[2]
  \begin{array}{ll}
  a, & H=S_3.\\[-1pt]
  a+b, & H=C_3\\[-1pt]
  a+c, & H=C_2\\[-1pt]
  a+b+2c,   & H=\{1\}.
  \end{array}
  \endgroup
$$
So, let $P_1,...,P_a$ be a basis of $E(K)\tensor\Q$ (the `trivial' part),
and complement it to a basis of $E(M)\tensor\Q$ with $Q_1,...,Q_b$
(the `sign' part) and to a basis of $E(L)\tensor\Q$ with $R_1,...,R_c$.
By clearing the denominators, we may assume that the $P$'s, $Q$'s and $R$'s
are actual points in $E(F)$.
If $g\in\Gal(F/K)$ is an element of order $3$, then
$$
  E(F)\tensor\Q=\langle P_1,...,P_a,Q_1...,Q_b,R_1,..,R_c,R_1^g,..,R_c^g\rangle,
$$
in other words these points form a basis of a subgroup of $E(F)$
of finite index. Now we can compute all the regulators.
Consider the three determinants
$$
  \cP = \det (\langle P_i, P_j \rangle_F), \quad
  \cQ = \det (\langle Q_i, Q_j \rangle_F), \quad\text{and}\quad
  \cR = \det (\langle R_i, R_j \rangle_F),
$$
where $\langle,\rangle_F$ is the N\'eron-Tate height pairing for $E/F$.
Recall that the pairing is normalised in such a way that it changes if
computed over a different field by the degree of the field extension.
For instance,
$$
  \langle P_i,P_j\rangle_K=\frac16 \langle P_i,P_j\rangle_F,
$$
so $\Reg_{E/K}$ is $(\frac16)^a \cdot \cP \cdot \square$, where $\square$
is the inverse square of the index of the lattice spanned by the $P_i$ in
$E(K)$. As we are only interested in regulators {\em modulo squares}, we may
ignore this. Similarly,
$$
  \Reg_{E/M} = (\frac13)^{a+b} \cP \cQ\cdot\square
  \qquad\text{and}\qquad
  \Reg_{E/L} = (\frac12)^{a+c} \cP \cR\cdot\square,
$$
because all $\langle P_i,Q_j\rangle_F$ and $\langle P_i,R_j\rangle_F$ are 0,
so both regulators are really products of two determinants.
(The height pairing is Galois-invariant, so different isotypical components
are always orthogonal to each other with respect to it.)
Finally, using Galois invariance again, together with the fact that
$R_i+R_i^g+R_i^{g^2}$ is $S_3$-invariant and so orthogonal to $R_j$, we find
that $\langle R_i,R_j^g\rangle_F=-\tfrac12\langle R_i,R_j\rangle_F$, so
$$
  \Reg_{E/F} = \cP \cdot \cQ \cdot \det \smallmatrix{A}{-\tfrac12 A}{-\tfrac12 A}{A} \cdot\square
            = \cP \cdot \cQ \cdot 3^c \cR^2\cdot\square,
$$
where $A$ is the matrix $(\langle R_i,R_j\rangle)_{i,j}$. Combining the four regulators yields
\beq
  \displaystyle
  \frac{\Reg_{E/F}\Reg_{E/K}^2}{\Reg_{E/M}\Reg_{E/L}^2}
    &\equiv&  \displaystyle \!\!\frac{ 3^c\cP\cQ\cR^2 \cdot (6^a\cP)^2 }{ 3^{a+b}\cP\cQ \cdot (2^{a+c}\cP\cR)^2 }
    \equiv 3^{a+b+c} \\[3pt]
    &\equiv& \!\!3^{a+(a+b)+(a+c)} \cr
    &\equiv& 3^{\rk E/K+\rk E/M+\rk E/L}\mod\Q^{\times2}.
\eeq
Together with \eqref{rc1}
this proves Theorem \ref{s3rksel}, assuming finiteness of $\sha$.

\subsection{Brauer relations and regulator constants}

The results of \S\ref{s:s3} generalise to arbitrary Galois groups
as follows (see \cite{Squarity,Tamroot} for details).

\begin{definition}
\label{grelation}
Let $G$ be a finite group, and write $\cH$ for the
set of representatives of subgroups of $G$ up to conjugacy.
A formal linear combination
$$
  \Theta = \sum\nolimits_i n_i H_i \qquad (n_i\in\Z,\> H_i\in\cH)
$$
is a {\em Brauer relation\/} if
$\oplus_i\C[G/H_i]^{\oplus n_i}=0$ as a virtual representation,
i.e. if the character
$\sum_i n_i \Ind_{H_i}^G\triv_{H_i}$ is zero.
\end{definition}

\begin{example}
\label{exdih}
The dihedral group $G=D_{2p}$ for an odd prime $p$ has a relation
$$
  \Theta = \>\{1\} - 2C_2 - C_p + 2\>G,
$$
the only one in $G$ up to multiples. For $p=3$ this is the relation
\eqref{eqs3rel}.
\end{example}

If $G=\Gal(F/K)$ is a Galois group and $E/K$ is an elliptic curve,
every Brauer relation $\Theta=\sum n_i H_i$ in $G$ gives an identity
of $L$-functions
$$
  \prod_i L(E/M_i,s)^{n_i}=1 \qquad\quad (M_i=F^{H_i}),
$$
and so a relation like \eqref{rc1} between regulators and Tamagawa numbers,
assuming that $\sha_{E/F}$ is finite. As in the case $G=S_3$, the
left-hand side $\prod_i\Reg_{E/M_i}{}^{n_i}$ depends only on
$E(F)\tensor\Q$ as a $G$-representation:

\begin{theorem}
\label{tregconst}
Let $G$ be a finite group and $\Theta=\sum n_i H_i$ a Brauer relation in~$G$.
Suppose $V$ is a $\Q G$-representation and $\lara: V\times V\to \R$ is a
non-degenerate bilinear $G$-invariant pairing\footnote
{It exists for every such $V$, as every rational representation is self-dual}.
Then
$$
  \cC_\Theta(V) := \prod_i \det \bigl(\tfrac{1}{|H_i|}\langle,\rangle\bigm| V^{H_i}\bigr)^{n_i}
$$
is a well-defined number in $\Q^\times/\Q^{\times 2}$,
and it is independent of the choice of the pairing $\langle,\rangle$.
\end{theorem}

The notation $\det \bigl(\tfrac{1}{|H_i|}\langle,\rangle\bigm| V^{H_i}\bigr)$
means the following: pick a basis $\{P_j\}$ of the invariant subspace $V^{H_i}$
and compute the determinant of the matrix whose entries are
$\tfrac{1}{|H_i|}\langle P_j, P_k\rangle$. Up to rational squares, this
is independent on the basis choice and the total product is
well-defined in $\R^\times/\Q^{\times 2}$. The theorem asserts that
it is in fact in $\Q^\times/\Q^{\times 2}$, and independent of
$\langle,\rangle$.

\begin{remark}
\label{anyk}
The theorem also holds with $\Q$ replaced by any other field $k$
where $|G|$ is invertible, $V$ by a self-dual $kG$-representation
and $\R$ by any field containing $k$.
\end{remark}

The proof of the theorem is reasonably straightforward
(see \cite{Squarity} \S2 or \cite{Tamroot}~\S2.ii), in the spirit
of what we did in \S\ref{s:s3} explicitly for $G=S_3$.
An immediate consequence is that if $\rho_1,...,\rho_k$
are all irreducible $\Q G$-representations, then the numbers
$\cC_\Theta(\rho_1),...,\cC_\Theta(\rho_k)$, called {\em regulator constants},
determine everything. In other words,
if we decompose $V=\bigoplus {\rho_j}^{\oplus a_j}$, then
$$
  \cC_\Theta(V) = \prod_j \cC_\Theta(\rho_j)^{a_j},
$$
as we can obviously pick a `diagonal' pairing on $V$ which respects the
decomposition.

\begin{example}
For $G=S_3$ and $\Theta=\{1\} - 2C_2 - C_3 + 2\>S_3$, we have
$$
  \cC_\Theta(\triv)=\cC_\Theta(\epsilon)=\cC_\Theta(\rho)=3 \in \Q^\times/\Q^{\times2}.
$$
(For example, if we choose the obvious trivial pairing on $\triv$, then
$$
  \cC_\Theta(\triv) = 1 \cdot (\tfrac12)^{-2}
    \cdot (\tfrac13)^{-1} \cdot (\tfrac16)^2 \equiv 3 \mod \Q^\times/\Q^{\times2},
$$
and similarly for $\epsilon$ and $\rho$.)
\end{example}

\begin{corollary}
\label{corregc}
Suppose $E/K$ is an elliptic curve, $F/K$ a Galois extension with
Galois group $G$ and $\Theta=\sum n_i H_i$ a Brauer relation in $G$.
Decompose $E(F)\tensor\Q=\bigoplus\rho_k^{\oplus a_k}$ into
irreducible $G$-representations. Then
$$
  \prod_i\Reg_{E/F^{H_i}}{}^{n_i} = \prod_k \cC_\Theta(\rho_k)^{a_k},
$$
\end{corollary}

\begin{proof}
Take $V=E(F)\tensor\Q$ and $\langle,\rangle$ the N\'eron-Tate height
pairing on $V$.
\end{proof}

For $G=S_3$, together with \eqref{rc1} this corollary reproves
Theorem \ref{s3rksel}, still assuming that $\sha_{E/F}$ is finite.

Finally, we discuss the modification necessary to turn this
into an unconditional statement about Selmer ranks.
Suppose $E/K$ is an elliptic curve, $\Gal(F/K)=G$ as before and $p$ is
a prime. The dual Selmer $\X_p(E/F)$ (cf.~Remark \ref{remrkp})
is a $\Q_p G$-representation which now plays a role analogous to the
$\Q G$-representation $E(F)\tensor\Q$.

\begin{theorem}
\label{selfd}
\par\noindent\par
\begin{enumerate}
\item
$\X=\X_p(E/F)$ is a self-dual $\Q_p G$-representation. In other words, it
possesses a $G$-invariant $\Q_p$-valued non-degenerate bilinear pairing.
\item
For any such pairing $\lara$
and a Brauer relation $\Theta=\sum n_i H_i$ in $G$, we have
$$
  \ord_p
  \Bigl(\prod_i \det\bigl(\tfrac{1}{|H_i|}\langle,\rangle\bigm|\X^{H_i}\bigr)^{n_i}\Bigr)
  \equiv
  \ord_p
  \prod_i C_{E/F^{H_i}}{}^{n_i}
  \mod 2.
$$
\end{enumerate}
\end{theorem}

For the proof see \cite{Selfduality}. In fact, for an explicit group
such as $G=S_3$ it is possible to give a direct proof of this,
by constructing an isogeny between the products of the relevant Weil
restrictions and applying Theorem \ref{thmQ};
this is how it is done in \cite{Squarity} Thm. 4.11 (with $p=3$).
Note that when $G=S_3$, every $G$-representation is rational,
so (1) is automatic.

As a corollary, we get Theorem \ref{s3rksel}.
(Decompose $\X_3(E/F)$ into $\Q_p S_3$-irreducibles and apply the theorem above.)
So we have proved the first equalities in Theorems \ref{i2isogthm}
and \ref{ithms3glo}, and it remains to prove the second ones.
For that we need to look carefully at the root numbers of elliptic curves
and carry out the local computations relating them to Tamagawa numbers.

\subsection{Parity in dihedral extensions}
\label{s:pardih}

The $S_3$-example generalises to dihedral groups. Suppose
$G=\Gal(F/K)=D_{2p}$, with $p$ an odd prime, and recall
from Example \ref{exdih}
that $G$ has a unique Brauer relation
$$
  \Theta = \>\{1\} - 2C_2 - C_p + 2\>G.
$$
Let $M$ and $L$ be the unique quadratic and one of the degree $p$ extensions
of $K$ in $F$, respectively. Apply Theorem \ref{selfd} and interpret the
left-hand side using regulator constants (using Theorem \ref{tregconst}
and Remark \ref{anyk} with $k=\Q_p$). Take any elliptic curve $E/K$,
and decompose
$$
  \X_p(E/F) = \triv^{\oplus n_\triv} \oplus \epsilon^{\oplus n_\epsilon}
              \oplus \rho^{\oplus n_\rho},
$$
where $\triv$ (trivial), $\epsilon$ (sign) and $\rho$ ($p-1$-dimensional)
are the distinct $\Q_pG$-irreducible representations; the latter decomposes
over $\bar\Q_p$ as a sum of 2-dimensionals,
$$
  \rho\tensor\bar\Q_p\quad\iso\quad\tau_1\oplus \cdots\oplus \tau_{\frac{p-1}2}.
$$
The regulator constants of $\triv, \epsilon$ and $\rho$ are easily seen
to be $p$ (as in the case $G=S_3$), and
Theorem \ref{selfd} gives the following parity statement:
$$
  n_\triv+n_\epsilon+n_\rho\equiv \ord_p \frac{C_{E/F}C_{E/K}^2}{C_{E/M}C_{E/L}^2}
   \mod 2.
$$
This can also be written as
$$
  \langle \triv+\epsilon+\tau_i, \X_p(E/F) \rangle
  \equiv \ord_p \frac{C_{E/F}C_{E/K}^2}{C_{E/M}C_{E/L}^2} \mod 2 \qquad \forall i,
$$
where $\lara$ stands for the usual inner product of characters of
representations.
In other words, the relation $\Theta$ gives a $p$-parity expression for
`the twist of $E$ by $\triv+\epsilon+\tau_i$' in terms of Tamagawa numbers
and periods.

\begin{remark}
For an alternative expression for exactly the same parity in terms of
`local constants', see Mazur and Rubin's papers \cite{MR, MR2};
the parity conjecture for these twists is now known for all elliptic curves
over number fields and all odd $p$, see \cite{Squarity,Kurast}
and de La Rochefoucauld \cite{dLR}.
\end{remark}

\subsection{The Kramer-Tunnell theorem}
\label{s:kratun}

We mentioned two types of relations between $L$-functions of
elliptic curves: one comes from a rational isogeny, and one
from Brauer relations in Galois groups. There is a third example, which
is classical: the relation for quadratic twists.

Suppose $M=K(\sqrt\beta)$ is a quadratic extension of number fields,
$E/K$ is an elliptic curve, and $E_\beta/K$ is the quadratic twist of $E$
by $\beta$:
$$
  E: y^2= x^3\!+\!ax\!+\!b, \qquad E_\beta: \beta y^2= x^3\!+\!ax\!+\!b
   \quad (\iso y^2= x^3\!+\!\beta^2ax\!+\!\beta^3b).
$$
The $l$-adic Tate modules of $E$ and $E_\beta$ are related by
$$
  T_l(E_\beta) = T_l(E)\tensor\epsilon,
$$
with $\epsilon: \Gal(M/K)\to\{\pm 1\}$ the non-trivial character, and Artin
formalism of $L$-functions (see \S\ref{s:root}) applied
to $\Ind_{\Gal(\bar K/M)}^{\Gal(\bar K/K)}\triv=\triv\oplus\epsilon$ yields a relation
$$
  L(E/M,s)=L(E/K,s)L(E_\beta/K,s).
$$
If we assume that $\sha_{E/M}$ is finite, Lemma \ref{lrel} gives
$$
  \frac{\Reg_{E/M}}{\Reg_{E/K}\Reg_{E_\beta/K}} \equiv
  \frac{C_{E/K}C_{E_\beta/K}}{C_{E/M}} \mod \Q^{\times 2},
$$
and it is easy to see that the left-hand side is
$2^{\rk E/M}$ (i.e. $2^{\rk E/K+\rk E_\beta/K}$) up to squares.
In fact, the Weil restriction $A=\Res_{M/K}E$ admits an isogeny
$$
  \phi: A \to E\times E_\beta,
$$
such that $\phi^t\phi=[2]$, and Theorem \ref{thmQ}
produces an unconditional version:
\beql{eqkratun}
  \rk_2 E/M \equiv \ord_2 \frac{C_{E/K}C_{E_\beta/K}}{C_{E/M}}\mod 2.
\eeql
This was used by Kramer \cite{Kra} and Kramer--Tunnell \cite{KT} to
prove the $2$-parity conjecture for $E/M$, by comparing the right-hand
side with the root number $w(E/M)$ by a local computation:

\begin{theorem}[Kramer--Tunnell]
\label{kratun}
Suppose the primes of additive reduction for $E$ above 2 are unramified
in $M/K$. Then the 2-parity conjecture holds for $E/M$:
$$
  (-1)^{\rksel EM2}=w(E/M).
$$

\end{theorem}

The restrictions on the reduction type can in fact be removed
using the methods of \S\ref{s:dtrf}.
We will not need this, so we refer the reader to
\cite{Kurast} (`proof of the Kramer-Tunnell conjecture').

\section{$L$-functions and root numbers}
\label{s:root}

To set up the notation, let $K$ be a number field,
and $\p$ a prime of $K$ with completion $K_\p$,
residue field~$\F_q$ of characteristic $p$ and uniformiser $\pi$.
We write $G_K$ for the absolute Galois group $\Gal(\bar K/K)$ of $K$,
and use a similar notation for other fields as well.
For the most of this section we work in the local setting, and we begin
by recalling the structure of the local Galois group at $\p$.
The reduction map on automorphisms puts $\GKp$ into an exact sequence
$$
  1\>\lar\> I_\p\>\lar\> \GKp\> {\buildrel{\hspace{-5pt}\mod\p}\over{\>\lar\>}}\>\> G_{\F_q} \>\lar\> 1,
$$
which defines the {\em inertia group} $I_\p$ at $\p$.
An {\em (arithmetic) Frobenius} is any element $\Frob_\p\in \GKp$ whose
reduction mod $\p$ is the map $x\mapsto x^q$.
If we choose an embedding $\bar K\injects \Kpbar$, then
$\GKp\injects G_K$ via the restriction map, and we can consider
$I_\p$ as a subgroup of $G_K$. Choosing a different embedding conjugates
$I_\p$, so it is really only well-defined up to conjugation.
Similarly, we can view $\Frob_\p$ as an element of $G_K$,
but it is only defined up to conjugation, and only modulo inertia.

We will often use two standard characters
that come the from identifications
$\GKp/I_\p\iso\prod\Z_l^\times$
and $I_\p/\vabove{\rm wild}{\rm inertia}\iso\prod_{l\ne p}\Z_l$.
Fix a prime $l\nmid p$.

\begin{definition}
The ($l$-adic) \emph{cyclotomic character} $\chi: \GKp \to \Z_l^\times$ is defined
by $\chi(I_\p)=1$ and $\chi(\Frob_\p)=q$. 
Alternatively, it is the action of $\GKp$ on the $l$-power roots of unity,
$$
  \chi:\> \GKp \>\lar\> \invlim[n] \Aut \mu_{l^n} \iso \invlim[n] (\Z/l^n\Z)^\times \iso
    \Z_l^\times.
$$
(In this way it can be defined as a character of $G_K$.)
\end{definition}

\begin{definition}
The ($l$-adic) \emph{tame character} $\phi: I_\p \to \Z_l$ is defined by
$$
  \phi(\sigma) = (\sigma(\pi^{1/l^n})/\pi^{1/l^n}) \in
    \invlim[n] \mu_{l^n} \iso \Z_l.
$$
\end{definition}

Recall that a $\GKp$-module $M$ is {\em unramified} if $I_\p$
acts trivially on $M$, and, similarly, a $G_K$-module is {\em unramified at $\p$}
if $I_\p\<G_K$ acts trivially on it.

\begin{example}
The trivial character $\triv$ and the cyclotomic character $\chi$
both give $\Z_l$ a structure of an unramified $\GKp$-module.
\end{example}

\subsection{$L$-functions}

Let $E/K$ be an elliptic curve.
For every prime $l$, the Galois group $G_K=\Gal(\bar K/K)$ acts on
the sets $E[l^n]=E(\bar K)[l^n]$ of $l^n$-torsion points of $E$ for all
$n\ge 1$. The fundamental arithmetic invariant of $E/K$ is its
{\em $l$-adic Tate module},
$$
  T_l E = \invlim[n] E[l^n],
$$
the limit taken with respect to the multiplication by $l$ maps
$E[l^{n+1}]\to E[l^n]$.
This is a free $\Z_l$-module of rank 2, and
$$
  V_l E = T_l E\tensor_{\Zl}\Ql
$$
is a 2-dimensional $\Ql$-vector space.
By the non-degeneracy and the Galois equivariance of the Weil pairing
on $E[l^n]$, we have $\det V_l E=\chi$.
The representation $V_l E$ (or sometimes its dual $V_lE^*$)
is called the \emph{$l$-adic representation} associated to $E/K$.
For varying $l$ these form a
{\em compatible system of $l$-adic representations}, meaning that they
satisfy two conditions:

\begin{enumerate}
\item All $V_l E$ are unramified at $\p$ for almost all primes $\p$ of $K$.
\item For such $\p$, the characteristic polynomial of Frobenius $\Frob_\p$
on $V_l E$ is independent of $l$ for $\p\nmid l$.
\end{enumerate}

To explain this, take a prime $\p$ of $K$ and any $l\nmid q$.
The criterion of N\'eron-Ogg-Shafarevich
(\cite{Sil1} Ch. VII) asserts that
\begin{center}
$V_l E$ is unramified at $\p$ $\quad\liff\quad$ $E$ has good reduction at $\p$.
\end{center}
In particular, this happens for almost all primes $\p$ of $K$ and this is
independent of $l$. This is condition (1).

Now, as $I_\p\normal \GKp$ and $V_l E$
is a $\GKp$-representation, the quotient $\GKp/I_\p$ acts on
the inertia invariants $(V_l E)^{I_\p}$, so we can talk of the action
of Frobenius $\Frob_\p$ on $(V_l E)^{I_\p}$, and this is independent
of the choice of $\Frob_\p$. Its characteristic polynomial
does not change under conjugation, therefore
it is completely choice-independent.
So we may define the local polynomial at $\p$,
$$
  F_\p(T) = \det\bigl( 1-\Frob_\p^{-1} T\bigm| (V_l E^*)^{I_\p} \bigr).
$$
(There are two technical points: we take the geometric Frobenius
$\Frob_\p^{-1}$, the inverse of the arithmetic one, and we also compute
the characteristic polynomial on the dual $V_l E^*$; both are just
standard conventions, and are not too important.)
As explained in \cite{Sil1} Ch. V,
\begin{center}
$E/K_\p$ has good reduction $\quad\implies\quad$
$F_\p(T)=1-a_\p T+q T^2$,
\end{center}
with $a_\p=q+1-|E(\F_q)|$. This polynomial has degree 2, since
$V_l E=(V_l E)^{I_\p}$ in the good reduction case.
Also, $F_\p(T)$ is in $\Z[T]$ rather than
just $\Z_l[T]$ and is independent of $l$.
This is condition (2).

\begin{definition}
The $L$-function of $E/K$ is a function of complex variable~$s$ given by
the Euler product
$$
  L(E/K,s) = \prod_\p F_\p(q_\p^{-s})^{-1}
  = \prod_\p
  \det\bigl( 1-q_\p^{-s}\Frob_\p^{-1} \bigm| (V_l E^*)^{I_\p} \bigr)^{-1}.
$$
Here the two products run over all primes $\p$ of $K$, and $q_\p$ denotes
the size of the residue field at $\p$.
(It follows from the Hasse-Weil bound that $L(E/K,s)$ converges
for $\Re s>3/2$.)
\end{definition}

In the same way one associates an $L$-function to any compatible system
$\rho=(\rho_l)_l$ of $l$-adic representations.
There are obvious notions of direct sums and induction for
$l$-adic representations and it is not hard to verify that their
$L$-functions satisfy the following:
\begin{itemize}
\item[(i)] If $\rho, \rho'$ are compatible systems of $l$-adic
representations of $G_K$, then
$$
  L(\rho\oplus\rho',s) = L(\rho,s)L(\rho',s).
$$
\item[(ii)] For $\rho$ as above and a subfield $F\subset K$,
$$
  L(\Ind_{G_K}^{G_F}\rho,s) = L(\rho,s).
$$
\end{itemize}
These properties are known as the {\em Artin formalism} for $L$-functions.

As for elliptic curves, we expect all $L$-functions of systems of $l$-adic
representations to have meromorphic (and usually analytic) continuation
to the whole of $\C$, and to satisfy a functional equation relating
the value at $s$ to that at $k-s$, where $k\in\Z$ is the {\em weight}
of $\rho$.

\begin{remark}
Most $L$-functions that we know of are supposed to arise
from $l$-adic representations. For instance, if $V/K$ is any non-singular
projective variety and $0\le i\le 2\dim V$, the \'etale cohomology groups
$H^i_{\text{\'et}}(V,\Q_l)$ form a compatible system. Thus we can
define the corresponding $L$-function $L(H^i(V),s)$, and it has weight $k=i+1$.
In this terminology,
$L(E/K,s)=L(H^1(E/K),s)$. (This the reason for taking the dual $V_l E^*$
instead of $V_l E$ in definition of $L(E/K,s)$.)

There is one subtle point though: for varieties other
than curves and abelian varieties, it is not known that the polynomials
$F_\p(T)$ are independent of $l$ when $\p$ is a prime of bad reduction.
This is conjectured to be true, and this
conjecture is implicitly assumed when one speaks of $L$-functions of
general varieties.
\end{remark}

\begin{remark}[Conductors]
Another arithmetic invariant that enters the functional equation of
any $L$-function is its conductor $N$. For an elliptic curve $E/K$,
the conductor of $L(E/K,s)$ is $N=\Delta_K^2\Norm_{K/\Q}(\cN_{E/K})$.
Here $\Delta_K$ is the discriminant of $K$, and
$\cN_{E/K}$ is the conductor of $E/K$. It is an ideal in $\cO_K$
(see \cite{Sil2}), and we interpret its norm to $\Q$ as a positive integer.
It is this $N$ that enters the functional equation in \S\ref{ssmainres}.
\end{remark}

For elliptic curves, let us determine $F_\p(T)$ in all cases, and verify
its independence of $l$ expicitly. There are several cases:

\subsubsection*{$E/K_\p$ has good reduction.}
As mentioned above, $F_\p(T)$ has degree 2, is independent of $l$ and
is of the form
$$
  F_\p(T) = 1-a_\p T+ q_\p T^2.
$$

\subsubsection*{$E/K_\p$ has split multiplicative reduction.}
In this case,
Tate's uniformization (`theory of the Tate curve') asserts
that there is a ($p$-adic analytic) isomorphism of $\GKp$-modules
$$
  E(\Kpbar) \iso \Kpbar^\times/a^\Z
$$
for some $a\in K_\p^\times$ with $v_\p(a)=-v_\p(j(E))>0$.
Writing $\zeta_{l^n}\in\Kpbar^\times$ for a primitive $l^n$th root of unity,
we get that in particular,
$$
  E[l^n] \>\>\iso\>\> \langle \zeta_{l^n}, a^{1/l^n} \rangle \subset \Kpbar^\times
$$
as a Galois module. In this basis, the action of Galois on $E[l^n]$ is
of the form
$$
  \GKp \ni \sigma \quad \longmapsto \quad
  \begin{pmatrix}
     \chi(\sigma) & v_\p(a)\!\cdot\!* \cr
         0        &   1 \cr
  \end{pmatrix}
$$
with $\chi$ the $l$-adic cyclotomic character  and `$*$'
restricting to the tame character $\phi$ on $I_\p$.
Passing to the inverse limit
$T_l E=\smash{\invlim}E[l^n]$ and tensoring with $\Q_l$ we find that the
actions on $V_l E$ and $V_l E^*$ are
$$
  \sigma\mapsto
  \begin{pmatrix}
     \chi(\sigma) &   * \cr
         0        &   1 \cr
  \end{pmatrix}
  \qquad\text{and}\qquad
  \sigma\mapsto
  \begin{pmatrix}
     \chi^{-1}(\sigma) &   0  \cr
         *             &   1 \cr
  \end{pmatrix}
$$
respectively.
The inertia $I_\p$ acts as $\smallmatrix 10\phi1$ on $V_lE^*$,
so is has a 1-dimensional invariant subspace, spanned by the second
basis vector. Frobenius acts trivially on it, so
$$
  F_\p(T) = \det\bigl( 1-\Frob_\p^{-1} T\bigm| (V_l E^*)^{I_\p} \bigr) = 1-T.
$$

\subsubsection*{$E/K_\p$ has non-split multiplicative reduction.}
Let $K_\p(\sqrt\xi)/K_\p$ be the unramified quadratic extension, and
$$
  \eta:\>\> \GKp \>\>\surjects\>\> \Gal(K_\p(\sqrt\xi)/K_\p) \>\>
  {\buildrel\sim\over\to} \>\> \{\pm 1\}
$$
the associated character of order 2. The quadratic twist
$E_\xi$ of $E$ has split-multiplicative reduction and
$V_l E = V_l(E_\xi)\tensor\eta$ as a Galois module. In other words,
$\GKp$ acts on $V_lE^*$ as
$$
  \sigma\mapsto
  \begin{pmatrix}
     \chi^{-1}(\sigma)\eta(\sigma) &   0  \cr
         *                 &   \eta(\sigma). \cr
  \end{pmatrix}
$$
Because $I_\p$ acts in the same way as in the split multiplicative
reduction case and $\eta(\Frob_\p)=-1$, we find that
the inertia invariants are again 1-dimensional and
$$
  F_\p(T) = 1+T.
$$

\subsubsection*{$E/K_\p$ has additive potentially multiplicative reduction.}

This is similar to the non-split multiplicative case, except that here
$K(\xi)/K$ is replaced by a ramified quadratic extension (namely,
the smallest extension where $E$ acquires split multiplicative reduction).
Now $I_\p$ acts through $\pm\smallmatrix 10*1$ on $V_l E^*$ with
non-trivial $\pm$, and has
therefore trivial invariants. So
$$
  F_\p(T) = 1.
$$

\subsubsection*{$E/K_\p$ has additive potentially good reduction.}

As in the last case, we claim
that $(V_l E^*)^{I_\p}=0$ and consequently
$$
  F_\p(T) = 1.
$$
Recall that $E$ acquires good reduction after a finite extension of $K_\p$.
By the N\'eron-Ogg-Shafarevich criterion,
$I_\p$ acts on $V_l E$ non-trivially and through a finite quotient.
If $V_l E$ had a 1-dimensional inertia invariant subspace,
then $I_\p$ would act as $\smallmatrix 1*0*$ in some basis.
The bottom right corner is 1 as $\det V_l E=\chi$ is unramified,
and the top-right corner is then 0 since an action
of a finite group is diagonalizable.
So $I_\p$ would act trivially, a contradiction.
Therefore $(V_l E^*)^{I_\p}=0$, as asserted.

\subsection{Weil-Deligne representations}

Whatever the reduction type of $E/K_\p$ is, note that $I_\p$ has an open
(i.e. finite index) subgroup which acts unipotently on $V_l E$, namely
trivially in the potentially good case
and as $\smallmatrix 1\phi01$ in the potentially multiplicative case.

Grothendieck's monodromy theorem asserts that
for any non-singular projective variety $V$ and any $i$, the action of 
some open subgroup of $I_\p$
on $\rho=H^i_{\text{\'et}}(V)$ is unipotent. It follows that it has the form
$1+\phi N$ for some fixed nilpotent endomorphism $N$ of $\rho$.
We are going to call such representations Weil representations ($N=0$) and
Weil-Deligne representations (any $N$).
There are two technical points: one is that we fix some embedding
$\Q_l\injects\C$ and make our representations complex instead of $l$-adic
from this point onwards; another one is that to get rid of the dependence
of the $l$-adic tame character $\phi$ of $l$, we just remember the nilpotent
endomorphism $N$ and how it commutes with the Galois group.

\begin{definition}
A {\em Weil representation} over $F$ of dimension $n$ is a homomorphism $G_F\to\GL_n(\C)$
whose kernel contains a finite index open subgroup of the inertia group
$I_{\bar F/F}$.
\end{definition}

\begin{definition}
A {\em Weil-Deligne representation} over $F$ is a
Weil representation $\rho: G_F\to\GL_n(V)$ together with a nilpotent
endomorphism $N\in\End(V)$ such that $\rho(g)N\rho(g)^{-1}=\chi(g)N$
for all $g\in G_F$.
\end{definition}

\begin{example}
If $E/K_\p$ is an elliptic curve, then $V_l E\tensor\C$ and
$(V_l E^*)\tensor\C$ have a natural structure of Weil-Deligne representations.
They are Weil
representations if and only if $E$ has integral $j$-invariant (equivalently,
$E$ has potentially good reduction).
\end{example}

\subsection{Epsilon-factors}

The current state of affairs is that we are very far from proving
the Hasse-Weil conjecture for compatible systems of $l$-adic representations
or even for elliptic curves over number fields. Even for elliptic curves
over $\Q$ the proof (via modularity) is rather roundabout. In some sense,
the only well-understood situation is the 1-dimensional case, that
of Hecke characters (a.k.a. `Gr\"ossencharakteren').
Also well-understood are the signs in the conjectural functional
equations for all $L$-functions. This is the theory of $\epsilon$-factors,
which we now sketch.

We will not need Hecke characters themselves, only their
local components.
Let $\p$ be a prime of $K$,
and denote by $F$ some finite extension of~$K_\p$.

\begin{definition}
A {\em quasi-character} of $G_F$ is a one-dimensional Weil-Deligne
(equivalently, Weil) representation. Alternatively, it is a
homomorphism
$$
  \psi: G_F\to\C^\times,
$$
which is continuous with respect to the profinite topology on $G_F$
and discrete topology on $\C^\times$.
\end{definition}

In his thesis, Tate associated to a quasi-character
its \emph{epsilon-factor} $\epsilon(\psi)\in\C^\times$, which enters
the functional equation for the local $L$-function of $\psi$:

\begin{notation}
\label{tatethesisnot}
Composing with the local reciprocity map $F^\times\to G_F^{ab}$,
consider $\psi$ also as a character $F^\times\to\C^\times$. Define
\beq
n(\psi) &=& \text{the conductor exponent of $\psi$}, \cr
b(F)    &=& v_\p(\Delta_{F/\Q_p}), \cr
h       &=& \text{any elt. of $F^\times$ of valuation $-n(\psi)-b(F)$, e.g.
            $\pi_F^{-n(\psi)-b(F)}$}, \cr
\epsilon(\psi) &=& \left\{\begin{array}{ll}
  \int_{h\O_F^\times}\psi(x^{-1})e^{2\pi i\Tr_{F/\Q_p}(x)} dx &
    \text{for $\psi$ ramified,}\cr
  \int_{h\O_F^\times}\psi(h^{-1}) dx = \frac{\psi(h^{-1})}{|h|_F}\int_{\O_F^\times}dx &
    \text{for $\psi$ unramified.}\cr
  \end{array}
\right.
\eeq
\end{notation}

The integrals are in effect finite sums (with the number of terms
growing with the conductor of $\psi$ and the ramification of
$F$ over $\Q_p$), so they can be explicitly computed for a given
quasi-character.

Tate's theory of signs in the functional equations extends uniquely
from quasi-characters to arbitrary Weil-Deligne representations
$$
  \rho: G_F \lar \GL_n(\C),
$$
for all finite $F/K_\p$, and all $n$:

\begin{theorem}[Langlands-Deligne \cite{DelC}]
There is a unique way to associate to each $\rho$ its
\emph{epsilon-factor} $\epsilon(\rho)\in\C^\times$ such that
\begin{enumerate}
\item (Multiplicativity.)
$\epsilon(\rho_1\oplus\rho_2)=\epsilon(\rho_1)\epsilon(\rho_2).$
\item (Inductivity in degree 0.)
If $\rho_1, \rho_2: G_F\lar\GL_n(\C)$ have the same degree, then%
\footnote{equivalently, $\epsilon(W)=\epsilon(\Ind W)$ for any virtual
representation $W$ of degree 0.}
$$
  \frac{\epsilon(\rho_1)}{\epsilon(\rho_2)} =
  \frac{\epsilon(\Ind_{G_F}^{\GKp}\rho_1)}{\epsilon(\Ind_{G_F}^{\GKp}\rho_2)}.
$$
\item (Quasi-characters.)
For quasi-characters $\psi: G_F\to\C^\times$
the $\epsilon(\psi)$ are as in \ref{tatethesisnot}.
\item (Semi-simplification.) Writing $\rho^{ss}$ for the semi-simplification
of $\rho$,
$$
  \epsilon(\rho)=\epsilon(\rho^{ss})\>
  \frac{\det(-\Frob_\p | (\rho^{ss})^{I_\p})}
       {\det(-\Frob_\p | \rho^{I_\p})}.
$$
\end{enumerate}
\end{theorem}

\begin{remark}
Uniqueness follows from the `Brauer induction' argument: every semisimple
Weil-Deligne representation is a $\Z$-linear combination of inductions
of quasi-characters. Existence is harder: one has to understand
`monomial relations' between inductions and use Stickelberger's theorem
to prove that the $\epsilon$-factors satisfy those relations.
\end{remark}

\begin{definition}
Write $\sgn z=z/|z|$ for the `sign' of $z\in\C^\times$ on the complex unit
circle. The \emph{local root number} of $\rho$ is defined as
$$
  w(\rho) = \sgn\epsilon(\rho) = \frac{\epsilon(\rho)}{|\epsilon(\rho)|}.
$$
\end{definition}

\begin{example}
For $\rho=\psi$ a 1-dimensional unramified quasi-character,
$$
  w(\psi) = \sgn\psi(h^{-1}) = \sgn \psi(\Frob_\p)^{b(F)}.
$$
In particular, the trivial representation has $w(\triv)=1$.
\end{example}

\begin{example}
Writing $q$ for the size of the residue field of $F$,
we find that the cyclotomic character also has
$$
  w(\chi)=(\sgn q)^{\ldots}=1.
$$
\end{example}

Because the $\epsilon$-factors are multiplicative in direct sums and inductive
in degree 0, clearly so are the root numbers; similarly,
$$
  w(\rho)=w(\rho^{ss})\>
  \frac{\sgn\det(-\Frob_\p | (\rho^{ss})^{I_\p})}
       {\sgn\det(-\Frob_\p | (\rho)^{I_\p})}
$$
as well. Here are some additional properties that are not hard to deduce:

\def\overarrow#1pt#2{\tbuildrel{#2}\over{\hbox to #1pt{\rightarrowfill}}}

\begin{proposition}[Tate \cite{TatN}, Deligne \cite{DelC}]
\noindent\par\noindent
\begin{itemize}
\item $w(\rho\oplus\rho^*)=(\det\rho)(-1)$, i.e. the image of $-1$ under
$$
  F^\times    \overarrow 50pt{loc. recip.}   G_F^{ab}
              \overarrow 50pt{$\det\rho$}    \C^\times.
$$
\item
$w(\rho_1\tensor\rho_2)=w(\rho_1)^{\dim\rho_2}\cdot\sgn(\det\rho_2)
(\pi_F^{n(\rho_1)+b(F)\dim\rho_1})$ if $\rho_2$ is unramified.
\end{itemize}
\end{proposition}

\subsection{Root numbers of elliptic curves}

\begin{definition}
Let $E/K_\p$ be an elliptic curve. Define its {\em local root number}
$$
  w(E/K_\p)=w(\rho); \qquad \rho=(V_l E)^*\tensor_{\Q_l}\C.
$$
\end{definition}
For elliptic curves (and, generally, for abelian varieties) this is known
to be independent of $l$ (with $\p\nmid l$) and
of the embedding $\Q_l\injects \C$, and it equals~$\pm 1$.

\begin{definition}
\label{defgloroot}
The {\em global root number\/} of an elliptic curve $E$ defined over
a number field $K$ is the product of the local root numbers
over all places~of~$K$,
$$
  w(E/K) = \prod_v w(E/K_v).
$$
The product is finite since $w(E/K_v)=1$ for primes of good reduction
(see below). Also, we let $w(E/K_v)=-1$ for all Archimedean~$v$.
\end{definition}

\begin{example}[Good reduction]
If $E/K_\p$ has good reduction, then $\rho$ is unramified by the
N\'eron-Ogg-Shafarevich criterion. Since $\det\rho$ is the cyclotomic
character,
$$
  w(E/K_\p) = w(1\tensor\rho)=w(\triv)^2\sgn\det\rho(\pi_F^{\cdots})=
            \sgn (q^{\cdots}) = +1.
$$
\end{example}

\begin{example}[Split multiplicative reduction]
If $E/K_\p$ has split multiplicative reduction, then
$$
  V_lE = \smallmatrix{\chi}{*}{0}{1}, \qquad
  \rho = \smallmatrix{\chi^{-1}}{0}{*}{1}, \qquad
  \rho^{ss} = \smallmatrix{\chi^{-1}}{0}{0}{1},
$$
and $\det\rho=\chi$ as before. Applying the semi-simplification formula,
we find
$$
  w(\rho) = w(\smallmatrix{\chi^{-1}}{0}{0}{1})\cdot
  \tfrac{\sgn\det(-\Frob_\p | \smallmatrix{\chi^{-1}}{0}{0}{1}) }
       {\sgn\det(-\Frob_\p | 1)}
  = \tfrac{\sgn\det\smallmatrix{-q^{-1}}{0}{0}{1}}{\sgn\det(1)} =
    \frac{1}{-1} = -1.
$$
\end{example}

\begin{example}[Non-split multiplicative reduction]
If $E/K_\p$ has non-split multiplicative reduction, the same computation
shows that $w(E/K_\p)=+1$.
\end{example}

To compute the root numbers of elliptic curves in the additive
reduction case one has to understand $\rho$ well enough.
One case when this works well is
when $E$ admits an isogeny, forcing the Galois action on $V_l E$ to be
less complicated than in general:

\subsection{Root numbers of elliptic curves with an $l$-isogeny}

As before, let $E/K_\p$ be an elliptic curve, and suppose $\p\nmid l$.
Assume that there is a degree~$l$ isogeny defined over $K_\p$,
$$
  \phi: E\lar E'.
$$
Recall that $\phi$ is a non-constant morphism of curves mapping
$O$ to $O$, and over $\Kpbar$ it is an $l$-to-$1$ map;
such a map automatically preserves addition on~$E$, and therefore maps
$E[l^n]$ to $E'[l^n]$, and hence $T_l E$ to $T_l E'$ as well.

\begin{notation}
Write $\ker\phi$ for the abelian group of points
in the kernel of~$\phi$ in $E(\Kpbar)$; thus $\ker\phi\iso C_l$
as an abelian group. Write also
$$
  F := K_\p(\text{coordinates of points in $\ker\phi$}) \subset K_\p(E[l]).
$$
\end{notation}

The points in $\ker\phi$ are permuted by $\GKp$, so $F/K_\p$ is a Galois
extension. This also means that the action of $\GKp$ on $E[l]$
is reducible, that is of the form $\smallmatrix **0*$. The Galois group
$\Gal(F/K_\p)$ is the image of
$$
  \GKp \overarrow 60pt{action on $E[l]$}
     \{\smallmatrix**0*\}
  \overarrow 60pt{top left corner}  \F_l^\times,
$$
so it is cyclic of order dividing $l-1$. 

\begin{notation}
Write $(-1,F/K_\p)$ for the Artin symbol
$$
  (-1,F/K_\p) = \bigleftchoice{+1}{\text{if $-1$ is a norm from $F$ to $K_\p$}}
                           {-1}{\text{otherwise.}}
$$
\end{notation}

The Artin symbol is the main class-field-theoretic invariant for cyclic
extensions, and the local root number of an elliptic curve with an isogeny
turns out to be closely related to it:

\begin{theorem}
\label{rootisogp}
If $E/K_\p$ has additive reduction, $\p\nmid l$ and $l\ge 5$, then
$$
  w(E/K_\p) = (-1,F/K_\p).
$$
\end{theorem}

\begin{proof}
We leave the potentially multiplicative case to the reader, and do the
(slightly more involved) case of potentially good reduction.

{\em Claim 1: $E[l]$ is unramified over $F$.}
Proof: Because $\det\rho=\chi$ and $G_F$ acts trivially on $\ker\phi$,
the image of $G_F$ in $\Aut(E[l])=\GL_2(\F_l)$ is contained in
$\smallmatrix 1*0\chi$. As $\chi$ is unramified,
the image $I$ of inertia $I_F$ is inside $\smallmatrix 1*01$, so
$|I|$ divides $l$. But $|I|$ divides 24 in the potentially good
case, so $I$ is trivial and $E[l]$ is unramified over $F$.

{\em Claim 2: $E[l^n]$ is unramified over $F$ for all $n\ge 1$.}
Proof: The image $I_n$ of $I_F$ in $\GL_2(\Z/l^n\Z)$ is of order dividing 24,
as before. But the kernel of $\GL_2(\Z/l^n\Z)\to\GL_2(\Z/l\Z)$ is of
$l$-power order, so this map is injective on $I_n$. But by step 1, the image
is trivial, so $I_n$ is trivial as well.

{\em Claim 3: $E/F$ has good reduction.}
Proof: N\'eron-Ogg-Shafarevich.

{\em Claim 4: The action of $\GKp$ on $V_l E$ is abelian.}
Proof: We want to show that the commutator subgroup $\GKp{}'$ acts trivially
on $V_l E$. Now,
\begin{itemize}
\item $\GKp{}'\subset \IKp$ as $\GKp/\IKp\iso\Gal(\kbar/k)\iso\hat\Z$ is abelian,
\item The image of $\GKp{}'$ in $\Gal(F/\Kp)$ is trivial as $\Gal(F/\Kp)$ is abelian,
\end{itemize}
so $\GKp{}'\subset I_F$. But $I_F$ acts trivially on $V_l E$ by Claim 2,
as asserted. 

{\em Step 5. The local root number.}
A semisimple abelian action can be diagonalised
over $\C$, so $\rho=\smallmatrix\psi00{\chi\psi^{-1}}$ for some
quasi-character $\psi$. Now we can compute
\beq
  w(\rho) &=& w(\psi)w(\chi^{-1}\psi)= w(\psi)w(\psi^*) & \text{(unramified twist formula)}\cr
          &=& w(\psi\oplus\psi^*)=\psi(-1) & \text{($\rho\oplus\rho^*$ formula)}\cr
          &=& \tilde\psi(-1) & \text{for any primitive character $\tilde\psi$}\cr
          &&&\text{of $\Gal(F/\Kp)$ that agrees}\cr
          &&&\text{with $\psi$ on inertia}\cr
          &=& (-1,F/\Kp). & \text{(local class field theory)}\cr
\eeq
\end{proof}

\begin{remark}
The theorem illustrates the fact that when $\GKp$ acts on $V_l E$ through an abelian
quotient, we have enough formulae to determine the local root number.
When the action is not abelian, the following result is very useful:

\begin{theorem}[Fr\"ohlich-Queyrut] Suppose $F(\sqrt\xi)/F$ is quadratic extension,
and $\psi: G_{F(\sqrt\xi)}\to \C^\times$ a quasi-character such that
$\psi|_{F^\times}=1$. Then
$$
  w(\psi)=\psi(\xi)\in\{\pm 1\}.
$$
\end{theorem}

Via the induction formula, the theorem computes the local root number of
the Galois representation $w(\Ind\psi)$. This is a 2-dimensional
representation which is not necessarily abelian, e.g. the Galois image
may be dihedral or quaternion (it has a cyclic subgroup of index 2).
This is enough to determine the local root numbers of all elliptic curves
over fields of residue characteristic at least 5 (i.e. $\p\nmid 2,3$);
see \cite{RohG}.

In residue characteristics 2 and 3 the situation is a bit more complicated;
see \cite{Hal} for $\Q_2$ and $\Q_3$, \cite{Kob} for $\p|3$ and
\cite{Root2,Whi} for $\p|2$, and also \cite{Kurast} for a general
root number formula; see also \cite{RohG, RohI} for the root numbers
of twists of elliptic curves.
\end{remark}

\section{Parity over totally real fields}

\def\Xp#1#2{\XX_p(#1/#2)}
\def\X#1#2#3{\XX_{#1}(#2/#3)}
\def\XX{{\mathcal X}}

We refer to Wintenberger \cite{Win} for the definition of modularity
for elliptic curves over totally real fields. The two propositions
below are essentially due to Taylor
(see \cite{TayO} proof of Thm. 2.4 and \cite{TayR} proof of Cor. 2.2).

%

\begin{proposition}
\label{taylor}
Let $F/K$ be a cyclic extension of totally real fields, and $E/K$ an
elliptic curve. If $E/F$ is modular, then so is $E/K$.
\end{proposition}

\begin{proof}
Let $\pi$ be the cuspidal automorphic representation associated to $E/F$.
Pick a generator $\sigma$ of $\Gal(F/K)$.
Then $\pi^\sigma=\pi$ and therefore by Langlands'
base change theorem
 $\pi$ descends to a
cuspidal automorphic representation $\Pi$ over $K$.
Associated to $\Pi$ there is a compatible
system of $\lambda$-adic representations $\rho_{\Pi,\lambda}$,
with $\lambda$ varying over the primes of some number field $k$
(see \cite{Win}).
Fix a prime $\lambda$ of $k$, let $l$ be the rational prime below it and
write $V_l(E/F)$ for the Tate module of $E/F$ tensored with $\Q_l$.
The restriction of $\rho=\rho_{\Pi,\lambda}$
to $\Gal(\bar F/F)$ agrees with
$V=V_l(E/F)\tensor_{\Q_l}k_\lambda$.
Because $E$ cannot have complex multiplication over $F$ (it is totally real),
$V_l(E/F)$ is absolutely irreducible, and therefore the only
representations that restrict to $V$ are
$W=V_l(E/K)\tensor_{\Q_l}k_\lambda$ and its twists by characters of
$\Gal(F/K)$.
Hence $\rho$ and $W$ differ by a 1-dimensional twist, whence $W$ is also
automorphic and $E/K$ is modular.
\end{proof}

%

\begin{proposition}
\label{redtomodular}
Let $E$ be an elliptic curve over a totally real field $K$, and $p$ a prime
number. If the $p$-parity conjecture for $E$ is true over every totally
real extension of $K$ where $E$ is modular, then it is true for $E/K$.
\end{proposition}

\begin{proof}
Because $E$ is potentially modular (see \cite{Win}, Thm. 1),
there is a Galois
totally real extension $F/K$ over which $E$ becomes modular.
By Solomon's induction theorem,
there are soluble subgroups $H_i\<G=\Gal(F/K)$ and integers $n_i$, such that
the trivial character $\triv_G$ can be written as
$$
  \triv_G=\sum_i n_i\Ind_{H_i}^G\triv_{H_i}.
$$

Write $K_i$ for $F^{H_i}$.
Since $\Gal(F/K_i)=H_i$ is soluble, a repeated application
of Proposition \ref{taylor} shows that $E/K_i$ is modular.
By Artin formalism for $L$-functions,
$$
  L(E/K,s)=\prod_i L(E/K_i,s)^{n_i}.
$$
On the other hand,
writing $\XX=\Xp E{F}$, for every $H\<G$ we have
$$
  \rksel E{K^H}p=
  \dim\XX^H=\blangle\XX,\triv_H\brangle_H
    =\blangle\XX,\Ind_H^G\triv_H\brangle_G.
$$
The first equality is \eqref{selinv} and the last one
is Frobenius reciprocity.

By assumption, the $p$-parity conjecture holds for $E/K_i$.
Therefore
\beq
  \rksel EKp
  &=&\displaystyle
  \blangle\XX,\triv_G\brangle_G =
  \sum\nolimits_i n_i\rksel E{K_i}p\\[4pt]
  &\equiv&\displaystyle\sum\nolimits_i n_i\ord_{s=1}L(E/K_i,s)
  =\ord_{s=1}L(E/K,s) \mod 2.
\eeq
\end{proof}

\begin{remark}
\label{remmod}
The proof also shows that $L(E/K,s)$ has a meromorphic continuation to $\C$,
with the expected functional equation.
This is the same argument as in \cite{TayR}, proof of Cor. 2.2.
\end{remark}

\begin{theorem}[\cite{Squarity,NekIV,Kurast}]
\label{2parthm}
Let $K$ be a totally real field, and $E/K$ an elliptic curve with
non-integral $j$-invariant.
Then the $p$-parity conjecture holds for $E/K$ for every prime $p$.
\end{theorem}

\begin{proof}
Let $\p$ be a prime of~$K$ with $\ord_\p j(E)<0$. If $E$ has additive
reduction at $\p$, it becomes multiplicative over some totally real quadratic
extension $K(\sqrt\alpha)$, and the quadratic twist $E_\alpha/K$
has multiplicative reduction at $\p$ as well.
Because
\beql{Kalpha}
\begin{array}{llllll}
  w(E/K(\sqrt\alpha))&=&w(E/K)w(E_\alpha/K)&&\text{and} \cr
    \rksel{E}{K(\sqrt\alpha)}p&=&\rksel EKp+\rksel {E_\alpha}Kp,
\end{array}
\eeql
it suffices to prove the theorem for
elliptic curves with a prime of multiplicative reduction.
Since multiplicative reduction remains multiplicative in all extensions,
by Proposition \ref{redtomodular} we may also assume that $E$ is modular.

By Friedberg-Hoffstein's theorem \cite{FH} Thm. B,
there is a quadratic extension
$M=K(\sqrt \beta)$ of $K$ with a prescribed behaviour
at a given finite set of primes of $K$,
such that the quadratic twist $E_\beta$ has analytic rank $\le 1$.
If we require that the multiplicative primes are unramified, then
$E_\beta$ also
has a prime of multiplicative reduction. By Zhang's theorem (\cite{ZhaH} Thm. A)
that generalises the Gross-Zagier formula over the rationals,
$\sha_{E_\beta/K}$ is finite and the Mordell-Weil rank of $E_\beta/K$ agrees
with its analytic rank;
in particular, the $p$-parity conjecture holds for $E_\beta/K$.
To prove it for $E/K$ we just need to show it for $E/M$ (cf. \eqref{Kalpha}).

Suppose $p=2$. Fix an invariant differential $\omega$ for $E/K$, and
choose $M$ so that all infinite places, all
bad places for $E$, and ones where $\omega$ is not minimal are split in $M/K$.
By \eqref{eqkratun} in \S\ref{s:kratun},
$$
  \rksel E{K(\sqrt\beta)}2 \equiv \ord_2 \frac{C_{E/K(\sqrt\beta)}C_{E_\beta/K}}{C_{E/K}}
    \equiv 0 \mod 2,
$$
since the other finite primes do not contribute to $C$. For the same reason
$w(E/M)=(\pm1)^2=1$; see \S\ref{s:predict} where this is spelled out
in more detail, or apply Theorem \ref{kratun}.
So the 2-parity holds for $E/M$, as required.

Finally suppose $p>2$. Then choose $M$ so that
(a) all bad places for $E$ except one multiplicative place $\p$ are split,
(b) $\p$ is inert, and
(c) all real places are inert, i.e. $M$ is totally complex.
Let $F_\infty/M/K$ be the $p$-adic anticyclotomic tower; thus
$F_\infty$ is the largest Galois extension of $K$ containing $M$ such that
$G=\Gal(F_\infty/K)$ is of the form $\Z_p^r\rtimes C_2$ with $C_2$ acting by $-1$.
The Artin representations of $G$ are $\triv$ and $\epsilon$ that factor
through $\Gal(M/K)$, plus two-dimensional ones of the form
$\rho=\Ind_M^K\chi$ for non-trivial $\chi: \Z_p^r\to\C^\times$.
Such a $\chi$ factors through a dihedral extension $F/K$ in $F_\infty$,
and by \S\ref{s:pardih} we have
$$
  \langle \triv+\epsilon+\rho, \X pEF\rangle \equiv \ord_p\frac{C_{E/F}}{C_{E/M}}
    \equiv 1 \mod 2,
$$
where the last equality again uses the fact that all bad primes for $E$
and all infinite primes split into pairs in $M$ (and therefore in $F$ as well),
except for $\p$ which gives the `1'.
For the same reason, we have
$$
  w(E/K)w(E/K,\epsilon)=w(E/M,\chi)=-1.
$$
We already know the $p$-parity conjecture for $E_\beta$,
$$
  w(E/K,\epsilon)=w(E_\beta/K)=(-1)^{\ord_2\rksel{E_\beta}K2}=(-1)^{\langle\epsilon,X_p(E/F)\rangle},
$$
so to show the $p$-parity conjecture for the trivial twist, we just need to
verify it for a {\em single} non-trivial $\chi$. Precisely such a character
is provided by the anticyclotomic theory:
a combination of Cornut--Vatsal's \cite{CV} Thm. 4.2 with
Nekov\'a\v r's \cite{NekE} Thm. 3.2
gives a $\chi$ such that the $\chi$-component
of $\X pEF$ has multiplicity 1, which is an odd number, as required.
\end{proof}

\begin{remark}
The same idea lies behind the proof of the $p$-parity conjecture over $\Q$,
see \cite{Mon} ($p=2$) and \cite{Squarity} (odd $p$) for details.
\end{remark}

\section{The $2$-isogeny theorem}
\label{s:2isog}

\def\m{m_F}

\begin{theorem}
\label{2isogthm}
Let $K$ be a number field and $E/K$ an elliptic curve with $E(K)[2]\ne 0$.
Then
$$
  (-1)^{\rk_2 E/K}=w(E/K).
$$
In particular, finiteness of $\sha$ implies the parity conjecture for $E$.
\end{theorem}

The proof will occupy the whole of \S\ref{s:2isog}.

\begin{notation}
\label{2isognot}
Fix a 2-torsion point $0\ne P\in E(K)[2]$, and let $\phi: E \to E'$
be a 2-isogeny whose kernel is $\{O,P\}$.
Furthermore, translate $P$ to $(0,0)$, so that $E$ and $E'$ become
\beq
 (E_{a,b}=)
   &&  E\phantom{'}:  & y^2 = x^3 + a x^2 + b x, && a,b\in \O_K\>,  \cr
   &&  E': & y^2 = x^3 - 2ax^2 + \delta x, && \delta=a^2-4b\>,
\eeq
with $\phi: E \to E'$ given by
$$
  \phi: (x,y) \mapsto (x+a+bx^{-1},y-bx^{-2}y) \>.
$$
\end{notation}

\begin{notation}
Denote
$$
  \sigma_\phi(E/K_v) =
  (-1)^{\ord_2 \frac{|\coker \phi_v|}{|\ker \phi_v|}}
  = (-1)^{1+\ord_2 |\coker \phi_v|},
$$
where $\phi_v$ is the induced map on local points
$\phi_v: E(K_v)\to E'(K_v)$.
(Note that
$\ker\phi_v=\{O,(0,0)\}$ is always of size 2.)
\end{notation}

\begin{strategy}
Recall from Theorem \ref{thmQ} and the definition of the global
root number \ref{defgloroot} that
\beql{2isogfor}
\displaystyle
  (-1)^{\rk_2 E/K} = \prod_v \sigma_\phi(E/K_v);\quad
  \prod_v w(E/K_v) = w(E/K),
\eeql
so a natural strategy is to make a term-by-term comparison at all places.
\end{strategy}

Write $F=K_v$ for some completion of $K$.

\subsection{Complex places}

Suppose $F=\C$. Because it is algebraically closed,
$\phi_v: E(\C)\to E'(\C)$ is surjective, that is $|\coker\phi_v|=1$.
Thus,
$$
  w(E/F)=-1=\sigma_\phi(E/F).
$$

\subsection{Real places}

\noindent\par\medskip\noindent
\begin{minipage}{7.3cm}
Suppose $F=\R$. It turns out, somewhat surprisingly,
that $w(E/\R)$ and $\sigma_\phi(E/\R)$ may not be equal:
$w(E/F)=-1$ as in the complex case but
$\phi_v: E(\R)\to E'(\R)$ is not always surjective.
For varying $a,b\in \R$ the picture is given on the right.
\end{minipage}\hfill
\smash{$
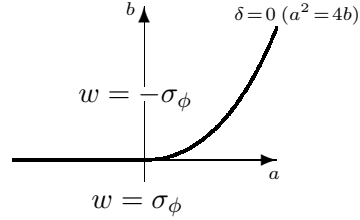
\begin{picture}(130,100)(0,65)
  \thicklines
  \put(0,50){\line(1,0){50}}
  \qbezier(50, 50)(80,50)(100,100)
  \linethickness{0.05mm}
  \put(0,50){\vector(1,0){100}}
  \put(50,78){\vector(0,1){30}}
  \put(50,42){\line(0,1){30}}
  \put(30,32){$w=\sigma_\phi$}
  \put(25,72){$w=-\sigma_\phi$}
  \put(97,43){\tiny $a$}
  \put(43,105){\tiny $b$}
  \put(84,103){\tiny $\delta\!=\!0\>(a^2\!=\!4b)$}
\end{picture}
$}

\bigskip
To see this, consider the structure of the group of real components
for $E$ and $E'$; recall that the group of real points is connected
if and only if the discriminant of the equation is negative:

\begin{itemize}
\item
If $(-2a)^2-4\delta=16b<0$, then $E'(\R)\iso S^1$, so $\phi_v$ is surjective.
\item
If $b>0$ and $\delta<0$ then $E(\R)\iso S^1$ and
$E'(\R)\iso S^1\times\Z/2\Z$, so $\phi_v$ is not surjective.
\item
If $b, \delta>0$, then $E(\R)\iso S^1\times\Z/2\Z\iso E'(\R)$.
Here $\phi_v$ is surjective if and only
if the points $\O,(0,0)$ of $\ker\phi$ lie on the same connected component.
(If they are on different components, the image of $\phi$ is connected;
otherwise, the identity component of $E(\R)$ maps 2-to-1 to the identity
component of $E'(\R)$, so the other component maps
to the other component since $\deg\phi=2$.)
So $\phi_v$ is surjective if and only if $0$ is the rightmost root of
$x^3+ax^2+bx$, and this is equivalent to $a$ being positive.
\end{itemize}

\subsection{The correction term}
To save our strategy, which seems to be somewhat in ruins, we have to introduce
a correction term that measures the difference between
$w(E/F)$ and $\sigma_\phi(E/F)$. It should be trivial at complex places
and depend on the signs of $a$, $b$ and $\delta$ at real places, so it is
natural to consider something like
$$
  (a,-b)_F(-a,\delta)_F
$$
where $(x,y)_F$ is the Hilbert symbol:

\begin{definition}
Let $F$ be a local field of characteristic zero. For $x,y\in F^\times$
the {\em Hilbert symbol\/} $(x,y)_F=\pm 1$ is
$$
  (x,y) = (x,F(\sqrt y)/F) = \leftchoice{+1}{\text{if $x$ is a norm from $F(\sqrt y)$ to $F$,}}
                                        {-1}{\text{otherwise.}}
$$
\end{definition}

Recall that $(x,y)$ is symmetric, bilinear as a map
$F^\times\times F^\times\to\{\pm1\}$, and
satisfies $(1-x,x)=1$ for $x\ne 0,1$.
In a field of residue characteristic $\ne 2$ it is explicitly
determined by

\smallskip
\begin{tabular}{ll}
$F=\C:$ & $(x,y)=1$ always; \cr
$F=\R:$ & $(x,y)=-1 \iff x,y<0$; \cr
$F/\Q_p$ finite, $p\ne 2$: & (unit, unit)$\,=1$,\cr
   & (uniformiser, non-square unit)$\,=-1$.
\end{tabular}

The suggested correction term $(a,-b)_F(-a,\delta)_F$ is trivial at
complex places and it is precisely set up to give the right signs
at the real places.
A few experiments suggest, that up to a missing factor of 2 in
front of $-a$ (invisible over the reals), this is the right
correction at {\em all} completions $F=K_v$, not just at infinite places.
One technical problem is that the Hilbert symbols do not make sense when $a=0$,
but this is easy to fix: if $|a|_F$ is small, then
\beq
  (a,-b)_F(-2a,a^2-4b)_F
    &=& (a,-b)_F(-2a,1-\tfrac{a^2}{4b})_F(-2a,-4b)_F \cr
    &=& (a,-b)_F(-2a,\square)_F(a,-4b)_F(-2,-4b)_F\cr
    &=& (-2,-b)_F,
\eeq
because elements of $F$ close to 1 are squares. So we may state

\begin{conjecture}
\label{2isogconj}
Let $F$ be a local field of characteristic zero, and $E/F$ an elliptic curve
with a 2-isogeny $\phi: E\to E'$ over $F$. In the notation of
\ref{2isognot},
$$
  w(E/F) = \sigma_\phi(E/F) \>\cdot\>
  \leftchoice
    {(a,-b)_F(-2a,a^2-4b)_F}{\text{if }a\ne 0}
    {(-2,-b)_F}{\text{if }a=0.}
$$
\end{conjecture}

This conjecture together with \eqref{2isogfor}
implies Theorem \ref{2isogthm}. Indeed, if $E$ as in \ref{2isognot}
is now defined over a number field $K$, the product formula for the Hilbert
symbol
$$
  \prod_v (x,y)_{K_v} = 1,\qquad \text{for all} \quad x,y\in K^\times
$$
implies that the correction term disappears globally.

Next, we prove the conjecture in a few cases.
As we dealt with infinite places already, assume from now on that
$F=K_v$ is non-Archimedean, i.e. a finite extension of $\Q_p$.
Also we may deal with the annoying exceptional case $a=0$:
if we prove the conjecture in all cases when $|a|_F$ is small,
then it holds when $a=0$ as well, because both the left- and the right-hand
side in the conjecture are continuous functions of the coefficients
$a$ and $b$ of $E$; we already proved this for the Hilbert symbols,
and for $C(E/F,\omega)$ and $w(E/F)$ this is a general fact:

\begin{proposition}
\label{continuity}
Suppose $E/F$ an elliptic curve in Weierstrass form,
$$
  E: y^2 + a_1 xy + a_3y = x^3 +a_2 x^2 + a_4 x + a_6, \qquad a_i\in F.
$$
There is an $\epsilon>0$ such that changing the $a_i$ to any $a'_i$
with $|a_i-a'_i|_F<\epsilon$ does not change the conductor,
minimal discriminant, Tamagawa number, $C(E,\frac{dx}{2y+a_1x+a_3})$, the root number of $E$
and the Tate module
$T_l E$ as a $\Gal(\bar F/F)$-module for any given $l\ne p$.
\end{proposition}

\begin{proof}
The assertion for the conductor, minimal discriminant,
Tamagawa number and $C$ follows from
Tate's algorithm (\cite{TatA} or \cite{Sil2} IV.9).
The claim for $T_l E$ is a result of Kisin (\cite{Kis} p. 569),
and the root number is a function of $V_l E=T_l E\tensor\Q_l$.
Alternatively, that the root number is locally constant can be proved
in a more elementary way: see Helfgott's \cite{Hel} Prop. 4.2
when $E$ has potentially good reduction; in the potentially multiplicative
case it follows from Rohrlich's formula (\cite{RohG}~Thm.~2(ii)).
\end{proof}

So assume from now on that $a\ne 0$.
Note also that our chosen model $y^2=x^3+ax^2+bx$ is unique up to
transformations $a\mapsto u^2a, b\mapsto u^4b$ for $u\in K^\times$.
As these do not change the Hilbert symbols $(a,-b), (-2a,a^2-4b)$ and
$(-2,-b)$, we may and will assume that the model is integral
and minimal when $v\nmid 2$.
Now we prove the conjecture in all cases when $E$ is semistable
with $v\nmid 2$; recall that $\ord_2 C_{E/F}=\ord_2 c(E/F)$
in this case, since the quotient $C/c$ is a power of the
residue characteristic.

\subsection{Good reduction at $v\nmid 2$}

Here $w(E/F)=1$, $\sigma_\phi(E/F)=1$ and $b,\delta\in\O_{F}^\times$.
If $a\in\O_{F}^\times$, then both the Hilbert symbols are $($unit,unit$)$,
hence trivial. For $a\equiv0\mod\m$, the expression
$-b\delta\equiv 4b^2\mod\m$ is a non-zero square mod $\m$,
so the product of the Hilbert symbols is again trivial.

\subsection{Split multiplicative reduction at $v\nmid 2$}
\label{s:splmul2}

Write $E/F$ as a Tate curve (\cite{Sil2} \S V.3)
$$
  E_q: y^2+xy = x^3+a_4(q)x+a_6(q), \qquad E(F)\iso F^\times/q^\Z,
$$
with $q\in m_F$ of valuation $v(q)=v(\Delta)$.  The coefficients have
expansions
$$
  a_4(q)=-5s_3(q), \quad a_6(q)=-\frac{5s_3(q)+7s_5(q)}{12}, \quad
    s_k(q)=\sum_{n\ge 1}\frac{n^k q^n}{1-q^n}\>,
$$
and they start
\beq
a_4(q)=-5q-45q^2-140q^3-365q^4-630q^5+O(q^6),\cr
a_6(q)=-q-23q^2-154q^3-647q^4-1876q^5+O(q^6)\>.
\eeq
The two-torsion, as a Galois set, is $\{1,-1,\sqrt{q},-\sqrt{q}\}$. For $u\ne 1$
in this set, the corresponding point on $E$ has coordinates
\beq
  X(u,q) = \frac{u}{(1-u)^2} + \sum\limits_{n\ge 1} \bigl(
    \frac{q^nu}{(1-q^nu)^2} + \frac{q^nu^{-1}}{(1-q^nu^{-1})^2} -2\frac{q^n}{(1-q^n)^2}
  \bigr),\cr
  Y(u,q) = \frac{u^2}{(1-u)^3} + \sum\limits_{n\ge 1} \bigl(
    \frac{q^{2n}u^2}{(1-q^nu)^3} + \frac{q^nu^{-1}}{(1-q^nu^{-1})^2} +\frac{q^n}{(1-q^n)^2}
  \bigr).
\eeq
We now have two cases to consider: the 2-torsion point
$(X(-1,q),Y(-1,q))\in E_q$ and (renaming $\pm\sqrt{q}$ by $q$)
the 2-torsion point $(X(q,q^2),Y(q,q^2))\in E_{q^2}$. In both cases,
we have $c(E)/c(E')=2^{\pm 1}$ and $w(E/F)=-1$, so we need
\beql{splitfor}
    (a,-b) \> (-2a,\delta) = 1\>,
\eeql
where $a, b, \delta$ are the invariants of the curve
transformed back to the form \ref{2isognot}
with the 2-torsion point at $(0,0)$.
First of all, $E_q$ has a model
$$
  y^2 = x^3+x^2/4+a_4(q)x+a_6(q)\>.
$$
Let $r=-X(u,q)$, and write $a_4=a_4(q), a_6=a_6(q)$. Then, after
translation, the curve becomes
$$
  E: y^2 = x^3+ax^2+bx, \quad a=1/4-3r, \quad b=2a_4-r/2+3r^2.
$$

Suppose we are in Case 1, so $r=-X(-1,q)$. Then the substitution
\beql{mulsub}
  x\to 4x-2r+1/2, \quad y\to 8y+4x
\eeql
transforms $E'$ into the form
$$
  E^\dagger: y^2+xy = x^3 + (-5q^2+O(q^4)) x + (-q^2+O(q^4))\>.
$$
We use the notation $O(q^n)$ to indicate a power series in $q$ with
coefficients in $\O_{F}$ that begins with $a_nq^n+...$.
In fact, $E^\dagger=E_{q^2}$ but we will not need this; it is only
important that it is again a Tate curve (in particular, this model
is minimal). From the expansions
$$
  r = 1/4+4 O(q), \quad
  a = -1/2+4 O(q), \quad
  b = 1/16+O(q),
$$
we have
$$
  (a,-b) = \text{(unit, unit)} = 1, \quad
  (-2a,\delta) = (1-8O(q),\delta) = (\square,\delta) = 1\>.
$$
Case 2 is similar: here
$$
  a = 1/4+2\, O(q), \quad
  b = q+O(q^2), \quad
  \delta = 1/16 + O(q)\>,
$$
so $a$ and $\delta$ are squares in $F$, and both Hilbert symbols
are therefore trivial.

\subsection{Nonsplit multiplicative reduction at $v\nmid2$}

Let $\eta\in O_F^\times$ be a non-square unit, so that $F(\eta)/F$
is the unramified quadratic extension of $F$.
Consider the twist of $E$ by $\eta$,
\beq
  E:      & y^2=x^3+ax^2+bx \>, \cr
  E_\eta: & y^2=x^3+\eta ax^2+\eta^2 bx \>.
\eeq
It has split multiplicative reduction, so by \eqref{splitfor}
$$
  (\eta a,-\eta^2 b) (2\eta a,\eta^2 (a^2-4b)) = 1.
$$
Now
\beq
  (\eta a,-\eta^2 b) = (\eta a,-b) = (\eta,-b) (a,-b), \cr
  (2\eta a,\eta^2 (a^2-4b)) = (2\eta a,a^2-4b) = (\eta,a^2-4b)(2a,a^2-4b),
\eeq
so comparing with the Hilbert symbols $(a,-b)$ and $(-2a,\delta)$
we have an extra term
\beql{2smcorr}
  (\eta,-b(a^2-4b)) = (\eta,-b^2\Delta(E')/16\Delta(E)) = (\eta,-\Delta(E')/\Delta(E)) \>.
\eeql
Because $x$ is a norm from $F(\eta)^\times$ to $F^\times$ if and only if $v(x)$ is even,
this Hilbert symbol is trivial precisely when
$v(\Delta(E'))\equiv v(\Delta(E))\mod 2$. From Tate's algorithm
(\cite{Sil2}, IV.9.4, Step 2),
$$
  c(E) = \left\{ \begin{array}{ll}
    1, & \text{$v(\Delta(E))$ odd},\cr
    2, & \text{$v(\Delta(E))$ even},
  \end{array}\right.
  \quad
  c(E') = \left\{ \begin{array}{ll}
    1, & \text{$v(\Delta(E'))$ odd},\cr
    2, & \text{$v(\Delta(E'))$ even},
  \end{array}\right.
$$
so the correction term \eqref{2smcorr} is trivial if and only
$c(E)/c(E')$ has even 2-valuation. This proves
Conjecture \ref{2isogconj} in the non-split multiplicative case,
when $v\nmid 2$.

\subsection{Deforming to totally real fields}
\label{s:dtrf}

There are several other reduction types, including all cases
when $v\nmid 2$ and when $E$ has good ordinary or multiplicative
reduction at $v|2$,
when Conjecture \ref{2isogconj} can be shown to hold directly \cite{Isogroot}.
It would be satisfying to do this in the remaining cases as well:

\begin{problem}
If $F/\Q_2$ is finite, and $E/F$ has good supersingular or
additive reduction, prove Conjecture \ref{2isogconj} directly.
\end{problem}

Having tried and failed to do this, we will complete the proof
of Conjecture \ref{2isogconj} and of Theorem \ref{2isogthm} by
a global argument. The idea is that if $E/K$ is an elliptic curve over a
number field and the conjecture holds for it at all primes but one,
the conjecture at the remaining prime is {\em equivalent} to the
2-parity conjecture for $E$. But there is a large supply of elliptic
curves $E$, namely those defined over totally real fields with
non-integral $j$-invariants, for which we know the 2-parity conjecture.
There are more than enough of these to approximate any given
elliptic curve over $F$.

\begin{assumption}
\label{2isogass}
From now on $F/\Q_2$ is finite, $a,b\in F$ satisfy
$b\ne 0, a^2-4b\ne 0$, and $E_{a,b}$ is an elliptic curve as in \ref{2isognot},
$$
  E_{a,b}: y^2 = x^3 + ax^2 + bx.
$$
\end{assumption}

\begin{lemma}
\label{2deflem}
There exists a totally real field $K$ with a unique place $v_0|2$, and
$\tilde a,\tilde b\in K$ such that
\begin{enumerate}
\item $K_{v_0}\iso F$.
\item Under this identification,
$|a-\tilde a|_{v_0}$ and $|b-\tilde b|_{v_0}$ are so small that
all terms in Conjecture \ref{2isogconj} are the same for $E_{a,b}/F$
and for $E_{\tilde a,\tilde b}/K_{v_0}$.
\item $E_{\tilde a,\tilde b}$ is semistable at all primes $v\ne v_0$ of $K$
      and has non-integral $j$-invariant.
\end{enumerate}
\end{lemma}

\begin{proof}
(1) Say $F=\Q_2[x]/(f)$ with monic $f\in\Q_2[x]$.
If $\tilde f\in \Q[x]$ is monic and $2$-adically close enough
to $f$, it defines the same extension of $\Q_2$ by Krasner's lemma. Now pick
such an $\tilde f$ which is also $\R$-close to any polynomial in $\R[x]$
whose roots are real (weak approximation), and set $K=\Q[x]/(\tilde f)$
and $v_0$ to be the prime above $2$ in $K$.

(2) Next, provided $\tilde a,\tilde b\in K$ are close enough to $a,b$
at $v_0$, the continuity of the Hilbert symbol (and the computation
preceding Conjecture \ref{2isogconj}) and Proposition \ref{continuity}
imply that (2) holds.

(3) As $v_0$ is the unique prime above 2, it suffices to guarantee that
$E_{\tilde a,\tilde b}$ is semistable at primes $v\nmid 2$ of $K$
and has $j(E)\not\in O_K$.
The curve has standard invariants
$$
  c_4 = 16(\tilde a^2-3\tilde b), \quad
  c_6 = -32\tilde a(2\tilde a^2-9b), \quad
  \Delta = 16\tilde b^2(\tilde a^2-4\tilde b).
$$
Pick any non-zero $\tilde a\in O_K$ which is close to $a$ at $v_0$.
Next, choose any $\tilde b\in O_K$ which is close to $b$ at $v_0$ and close
to 1 at all primes $v\ne v_0$ that divide $\tilde a$.
Then either $\Delta$ or $c_4$ is a unit at every prime $v\ne v_0$,
and this ensures that
$E$ is semistable outside $v_0$ (\cite{Sil1} Prop. VII.5.1).
If, in addition, we force $b$ to be divisible by at least one prime
$\p\nmid 2$ (weak approxmation all the time), this guarantees that
$j(E)$ is non-integral at $\p$.
\end{proof}

\begin{claim}
Conjecture \ref{2isogconj} is true.
\end{claim}

\begin{proof}
First we finish off the case when $F$ has residue characteristic 2.
Let $F/\Q_2$ and $E_{a,b}$ be as above,
and $K$ and $E=E_{\tilde a,\tilde b}$ as in the lemma.
We want to show that Conjecture \ref{2isogconj} holds for $E_{a,b}/F$,
equivalently for $E/K_{v_0}$.

Because $K$ has a unique prime $v_0$ above 2, and $E$ is semistable at all
other primes, Conjecture \ref{2isogconj} holds for $E/K_v$ for all
$v\ne v_0$. In view of \eqref{2isogfor}, the conjecture at $v_0$
is equivalent to the 2-parity conjecture for $E/K$. But $E$ has non-integral
$j$-invariant and $K$ is totally real, so the latter holds.

Finally, suppose $F$ has odd residue characteristic $p$, and $E/F$ has
additive reduction. Exactly as in Lemma \ref{2deflem},
find a totally real field $K$ with a place $v_0|p$ and $\tilde E/K$ which
is $v_0$-adically close to $E$, and semistable at all $v\ne v_0$.
Again, the 2-parity conjecture for $\tilde E/K$ together with
Conjecture \ref{2isogconj} at all $v\ne v_0$ (now including $v|2$)
proves Conjecture \ref{2isogconj} at $v_0$ as well.
\end{proof}

As explained above, reversing the argument yet again proves Theorem~\ref{2isogthm}.

\section{The $p$-isogeny conjecture}
\label{s:pisog}

What is the analogue of the results of the previous section
for an isogeny whose degree is an odd prime $p$?
Let $\phi: E\to E'$ be such an isogeny, with $E, E'$ and $\phi$ all
defined over $K$. Again, write
$$
  \sigma_\phi(E/K_v) =
  (-1)^{\ord_p \frac{|\coker \phi_v|}{|\ker \phi_v|}},
$$
so that by Theorem \ref{thmQ},
\beql{pisogfor}
\displaystyle
  (-1)^{\rk_p E/K} = \prod_v \sigma_\phi(E/K_v);\quad
  \prod_v w(E/K_v) = w(E/K).
\eeql
Again, let us make a term-by-term comparison at all places.
Fix a place $v$ of $K$.
The difference between $w(E/K_v)$ and $\sigma_\phi(E/K_v)$ is suggested
by the root number computation of Theorem \ref{rootisogp}, and we formulate
the following

\begin{conjecture}
\label{pisogconj}
Let $\cK$ be a local field of characteristic zero, and $E/\cK$ an elliptic curve
with a $p$-isogeny $\phi: E\to E'$ over $\cK$. Write $\cF=\cK(\ker\phi)$.
Then
$$
  w(E/\cK) = \sigma_\phi(E/\cK) \> (-1,\cF/\cK).
$$
\end{conjecture}

If this conjecture is true, applying it for $E/K$ at completions and
taking the product over all places, we get the $p$-parity conjecture
for $E/K$. Let us verify that the conjecture holds in most situations:

If $\cK=\C$ then $|\ker\phi|=p$, $|\coker\phi|=1$, $w(E/\cK)=-1$
and $\cF=\cK$, so the formula holds.
The same is true if $\cK=\R$ and the points in $\ker\phi$ are real.

If $\cK=\R$ and the points in $\ker\phi$ are not defined over $\R$,
then then $|\ker\phi|=1$, $|\coker\phi|=p$ and $\cF=\cC$,
so the correction term $(-1,\C/\R)=-1$ compensates for the kernel,
and the formula holds again.

Suppose $\cK$ is a finite extension of $\Q_l$ for $l\ne p$.
The standard exact sequences for
$E_0(\cK)\subset E(\cK)$ and $E_1(\cK)\subset E_0(\cK)$ (\cite{Sil1} Ch. VII)
show that
$$
  \sigma_\phi(E/\cK) = (-1)^{\ord_p\frac{c(E/\cK)}{c(E'/\cK)}},    
$$
where $c$ is the local Tamagawa number.
(Use that $E$ and $E'$ have the same reduction type and the same number of
points over the residue field, and that $[p]$ is an isomorphism on the
formal groups when $p\ne l$.) Now we have a few cases:

If $E/\cK$ has good reduction, then $\cK(E[p])/\cK$ is unramified by the
N\'eron-Ogg-Shafarevich criterion, in particular $\cF/\cK$ is unramified as
well. So all units are norms in this extension by class field theory,
in particular $-1$ is. The Tamagawa numbers are trivial as well, so
all terms in the conjecture are 1.

If $E/\cK$ has multiplicative reduction, then $\cF/\cK$ is still unramified,
since the inertia in $\cK(E[p])/\cK$ is pro-$p$ in the multiplicative case
and $[\cF:\cK]|(p-1)$ is coprime to $p$; so the Artin symbol is still 1.
Now $w(E/\cK)=-1$ as opposed to the the good case, but also
the quotient $\frac{c(E/K)}{c(E'/K)}$ is either $p$ or $1/p$
(same argument as in \S\ref{s:splmul2} for $p=2$), and the conjecture holds.

If $E/\cK$ has additive reduction and $p,l$ are at least 5,
then the Tamagawa numbers have order $\le 4$, coprime to $p$.
So the conjecture claims that $w(E/K)=(-1,\cF/\cK)$, and this is precisely
what we proved in Theorem \ref{rootisogp}; in fact, it is this computation
that suggests the correction term in the conjecture.

When $p=l$, the problem becomes trickier. It is still manageable when
$E/\cK$ is semistable, by a careful analysis of $\sigma_\phi$
which also involves the action of $\phi$ on the formal groups
(see \cite{Isogroot}). For $p=3$ the whole conjecture can be
settled by a deformation argument, exactly as we did for $p=2$
(see \cite{Kurast}; presumably the argument there also works
when $p=5$ and 7, when $X_0(p)$ still has genus 0).
The remaining problem is

\begin{problem}
Prove Conjecture \ref{pisogconj} when $p>7$, the field $\cK$ has residue
characteristic $p$ and $E/\cK$ has additive reduction.
\end{problem}

Coates, Fukaya, Kato and Sujatha \cite{CFKS} have proved this conjecture
when $E/\cK$ acquires semistable reduction after an abelian extension
of $\cK$, with a `tour de force' computation with crystalline cohomology.
Moreover, they proved it for arbitrary principally polarised
abelian varieties; $p$-isogeny has to be replaced by a $p^g$-isogeny
with a totally isotropic kernel. It would be very interesting to find
at least a conjectural generalisation for $p=2$:

\begin{problem}
Suppose $\cK$ is a local field of characteristic 0, and $\phi: A\to A'$ is an
isogeny of principally polarised abelian varieties over $K$
such that $\phi^t\phi=[2]$ (up to isomorphisms given by the polarisations).
Find a relation between $\sigma_\phi(A/\cK)$ and $w(A/\cK)$.
\end{problem}

This would have important parity implications for abelian varieties.
We also mention here that Trihan and Wuthrich \cite{TW} have proved an
analogue of Conjecture \ref{pisogconj} over function fields of characteristic
$p$, when $\phi$ is the dual of the Frobenius isogeny. Because this
isogeny exists for elliptic curve over a function field, they thus proved
the $p$-parity conjecture in characteristic $p$.

Finally, returning to Conjecture \ref{pisogconj} itself, it would also
be very interesting to see whether it can be settled by a deformation
argument, like we did it for $p=2$. The problem is that because $X_0(p)$
has positive genus for large $p$, it is unclear how to deform
elliptic curves with a $p$-isogeny while keeping control over other places
of the resulting totally real field. It is known how to construct points
of essentially arbitrary varieties over totally real fields
(see e.g. \cite{Pop}),
but it is not clear whether these results are sufficient in this case.

\section{Local compatibility in $S_3$-extensions}

To complete the proof of Theorem \ref{imain} it remains to settle
the second equality in Theorem \ref{ithms3glo}.
Thus, suppose $E/K$ is an elliptic curve, $F/K$ is a Galois $S_3$-extension
and $M/K$ and $L/K$ are the quadratic and cubic intermediate extensions.
We want to show that
\daggerequation{s3glo}{\hbox{$\dagger_{\text{\smaller[6]glo}}$}}{
  w(E/K)w(E/M)w(E/L) = \smash{(-1)^{\ord_3\frac{\CC EF\CC EK^2}{\CC EM\CC EL^2}}}.
}%
Both sides are products of local terms, and,
as in the isogeny case,
we want to compare the contributions above each place of $K$.

Fix such a place $v$ of $K$, and denote by $w_K$, $w_M$, $w_L$ and $w_F$
the product of the local root numbers of $E$ over the places of $K$, $M$, $L$
and $F$ above $v$. Fix also an invariant differential $\omega$ for $E/K$ and
write $C_K$, $C_M$ etc. for the product of the modified local Tamagawa numbers
of $E$ over the places of $K$, $M$ etc., computed using $\omega$
(see end of Notation \ref{mainnot}).
We need to show the following local statement:
\daggerequation{s3loc}{\hbox{$\dagger_{\text{\smaller[6]loc}}$}}{
  w_K w_M w_L = \smash{(-1)^{\ord_3\frac{C_F C_K^2}{C_M C_L^2}}}.
}%

\begin{lemma}
\label{s3loclem}
If there is more than one place above $v$ in $F$ or
$E$ has good reduction at $v$ then \eqref{s3loc} holds.
\end{lemma}

\begin{proof}
It is easy to see that for $v|\infty$ both the left-hand side and the
right-hand side are trivial%
\footnote{This is true in any Brauer relation, see e.g \cite{Tamroot} Proof of Cor. 3.4},
so suppose $v$ is a finite place.
Then the right-hand side is the same as
$\ord_p(C_F/C_M)\!\mod 2$. Now consider the following cases.

Case 1: $v$ splits in $M/K$. (In particular, this must happen if
$F/M$ is inert above $v$.) As the number of
primes above $v$ in $M$ is even, $C_F$ and $C_M$ are both squares,
and $w_M=1$.
Since in this case $L_v/K_v$ is Galois of odd degree, $w_L=w_K$ by
Kramer--Tunnell \cite{KT}, proof of Prop. 3.4.

Case 2: the prime above $v$ in $M$ splits in $F/M$.
Then $C_F=C_M{}^3$, so $C_F/C_M$ is a square.
Under the action of the decomposition group $D_v$ at $v$, the
$G$-sets $G/\Gal(F/L)$ and $(G/\Gal(F/M))\smallcoprod(G/G)$ are isomorphic.
So the number of primes above $v$ with a given ramification and
inertial degree is the same in $L$ as in $M$ plus in~$K$. It follows that
the local root numbers cancel, $w_Kw_Mw_L=1$.

Case 3: $F/M$ is ramified above $v$ and $E$ is semistable at $v$.
The contributions from $\omega$ cancel modulo squares,
and $w_K=w_L$.
If $E$ has split multiplicative reduction over a prime of $M$ above $v$,
this prime contributes $p$ to $C_F/C_M$ and $-1$ to the root number.
If the reduction is either good or non-split, it contributes to neither.
\end{proof}

\begin{proposition}
The formula \eqref{s3loc} holds in all cases.
\end{proposition}

\begin{proof}
The only case not covered by the lemma above is when $E/K$ has
additive reduction at $v$ and $F/K$ has a unique prime $\tilde v$ above $v$.
We will use a continuity argument to settle this case.

Pick an $S_3$-extension $\cF/\cK$ of totally real number fields with
completions $\cK_{w}=K_v$ and $\cF_{\tilde w}=F_{\tilde v}$ for some prime
$\tilde w|w$ in $\cF/\cK$ (same argument as in Proposition \ref{continuity}).
Choose an elliptic curve $\E/\cK$ which is close enough $w$-adically to $E/K_v$,
with semistable reduction at all places $\ne w$ where $\cF/\cK$ is ramified
and at least one prime of multiplicative reduction.
By `close enough' we mean that the left- and the right-hand sides of
\eqref{s3loc} are the same for $E/K_v$ and $\E/\cK_w$
(Proposition \ref{continuity}).

By the 3-parity conjecture for $\E$
over the intermediate fields of $\cF/\cK$ (Theorem \ref{2parthm})
and Theorem \ref{s3rksel}, we find that \eqref{s3glo} holds.
Since the terms in it agree at all places except possible $w$ by
Lemma \ref{s3loclem},
they must agree at $w$ as well.
This proves \eqref{s3loc} for $\E/\cK_w$ and hence for $E/K_v$ as well.
\end{proof}

This completes the proof of Theorem \ref{imain}.

\section{Parity predictions}
\label{s:predict}

The purpose of this final section is to collect some peculiar predictions of
the parity conjecture concerning ranks of elliptic curves over number fields.

\begin{definition}
Write $H(\tfrac pq)=\max(|p|,|q|)$ for the usual `naive height' of a rational
number. We say that a subset $S\subset\Q$ has density $d$ if
$$
  \lim_{X\to\infty}\, \frac {\#\{a\in S \,|\, H(a)<X \}}{\#\{a\in\Q \,|\, H(a)<X \}} = d.
$$
\end{definition}

There is a folklore `minimalistic conjecture':

\begin{conjecture}
Let $E/\Q(t)$ be an elliptic curve of Mordell-Weil rank $r$.
For $a\in\Q$ write $E_a$ for its specialisation $t\mapsto a$. Then
$$
  \rk E_a/\Q = \bigleftchoice{r}{\text{if }\>w(E_a/\Q)=(-1)^r}{r+1}{\text{otherwise}}
$$
for a set of rational numbers $a$ of density 1.
\end{conjecture}

Note that $E_a$ is indeed an elliptic curve for all but finitely many $a$,
so the conjecture makes sense. What is says is that generically, the
rank of a fiber is a sum of the `geometric' contribution from the
points in the family plus an `arithmetic' contribution from the root number.

\subsection{Semistable curves in cubic extensions}

The simplest elliptic curves are {\em semistable\/} ones,
and their root numbers are particularly nice:

\begin{definition}
An elliptic curve over $E/K$ is {\em semistable\/} if it has good
or multiplicative reduction at all primes of $K$.
\end{definition}

Because places of good and non-split multiplicative reduction do not
contribute to the global root number, and infinite and split multiplicative
places contribute $-1$, for semistable $E/K$,
$$
  w(E/K) = (-1)^{\#\{v|\infty\}} (-1)^{\#\{\text{$v$ split multiplicative for $E$\}}}.
$$
Thus, the parity conjecture implies

\begin{conjecture}
If $E/\Q$ is semistable, then
$$
  \rk E/\Q \>\equiv\> 1 \>+\> \#\{\text{\smaller[1]primes $p$ where $E$ has split mult. red.\}}\mod 2.
$$
\end{conjecture}

\begin{example}[\cite{VD}]
Take $E=19A3$ over $\Q$,
$$
  E: y^2+y=x^3+x^2+x, \qquad \Delta=19, \>\>\text{split mult. at 19}.
$$
Two-descent shows that its rank over $\Q$ is 0 (in accordance with the
prediction that it is even).

Now take $K_m=\Q(\sqrt[3]m)$ for some cube-free integer $m>1$. Then the
conjecture asserts that
$$
\begin{picture}(200,28)
  \put(0,18){$\rk E/K_m \equiv \>2 \>\>+\>\>  \leftchoice
    {3}{\text{19 splits in $K_m$}}{1}{\text{otherwise}}  \equiv 1\mod 2.$}
  \put(47,02){$\scriptscriptstyle\#\{v|\infty\}$}
  \put(107,02){$\scriptscriptstyle\#\{v|19\}$}
\end{picture}
$$
So $\rk E/K_m$ should always be odd, in particular positive. We get
the following surprising statement:

\begin{conjecture}
For every $m>1$, not a cube, the equation $y^2+y=x^3+x^2+x$ has infinitely
many solutions in $\Q(\sqrt[3]m)$.
\end{conjecture}

It is known that an elliptic curve over $\Q$ acquires rank over
infinitely many fields of the form $\Q(\sqrt[3]m)$ (\cite{Cubic} Thm. 1),
but the full conjecture appears to be completely unapproachable at the
moment.
\end{example}

\subsection{Number fields $K$ such that $w(E/K)=1$ for all $E/\Q$}
\noindent\par\noindent
Let $K=\Q(\sqrt{-1},\sqrt{17})$. This field has a peculiar property
that every place of $\Q$ splits into an even number of places in it (2 or 4):
\begin{itemize}
\item $K$ has 2 (complex) places $v|\infty$.
\item 2 splits in $\Q(\sqrt{17})$, and thus in $K$ as well.
\item 17 splits in $\Q(\sqrt{-1})$, and thus in $K$ as well.
\item Primes $p\ne 2,17$ are unramified in $K/\Q$, so their
decomposition groups are cyclic, $D_p\ne \Gal(K/\Q)=C_2\times C_2$;
so such $p$ split as well.
\end{itemize}
Thus, for {\em any} elliptic curve $E/\Q$,
$$
  w(E/K) = \prod_\p w(E/K_\p) = \prod_p (\pm 1)^{\text{even}} = +1,
$$
and we get
\begin{conjecture}
Every elliptic curve $E/\Q$ has even rank over $\Q(\sqrt{-1},\sqrt{17})$.
\end{conjecture}

The existence of such number fields was pointed out to us by Rubin.
Note also that the same conjecture may be stated for abelian varieties.
As an exercise, we leave it to the reader to use the same ideas to show
that the parity conjecture implies the following:

\begin{conjecture}
Every elliptic curve over $\Q$ with split multiplicative reduction at 2
has infinitely many rational points over $\Q(\zeta_8)$.
\end{conjecture}

\subsection{Goldfeld's conjecture over $\Q$}

\begin{definition}
For an elliptic curve
$$
  E/\Q:\>\> y^2 = f(x)
$$
and a (usually square-free) integer $d$, the {\em quadratic twist\/}
of $E$ by $d$ is
$$
  E_d/\Q:\>\> d y^2 = f(x).
$$
\end{definition}
Note that $E\iso E_d$ over $\Q(\sqrt d)$, but not over $\Q$.
Now, if $d_0<0$ is such that all primes $p|2\Delta_E$ split in
$\Q(\sqrt{d_0})$, then it is easy to see that
$$
  w(E_{dd_0}/\Q)=-w(E_d/\Q)\qquad\text{for all square-free $d$.}
$$
In other words, the involution $d\!\leftrightarrow\!dd_0$ on $\Q^\times/\Q^{\times2}$
changes the sign of $w(E_d)$. So

\begin{tabular}{l}
$w(E_d/\Q)=+1$ for 50\% square-free $d$'s,\cr
$w(E_d/\Q)=-1$ for 50\% square-free $d$'s,\cr
\end{tabular}

\noindent
meaning that
\begin{equation}
\label{density}
  \frac{\#\{\text{
    $|d|\le X$ square-free $| w(E_d/\Q)=1$
  }\}}{\#\{\text{
    $|d|\le X$ square-free
  }\}} \lar \frac 12 \quad\text{as $X\to\infty$}.
\end{equation}
The `minimalistic conjecture' above becomes a famous conjecture of Goldfeld:
\begin{conjecture}[Goldfeld]
Let $E/\Q$ be an elliptic curve. Then

\begin{tabular}{ll}
$\rk E_d/\Q=0$    & for 50\% square-free $d$'s,\cr
$\rk E_d/\Q=1$    & for 50\% square-free $d$'s,\cr
$\rk E_d/\Q\ge 2$ & for \phantom{5}0\%  square-free $d$'s.\cr
\end{tabular}
\end{conjecture}

\noindent
Note that `$0\%$' does not exclude the possibility of $E$ having infinitely
many quadratic twists of rank $\ge 2$. It only says that
$$
r_{\ge 2}(X):=
\#\{\text{
    $|d|\le X$ square-free $| \rk E_d/\Q\ge 2$
  }\}
$$
is $o(X)$ (the denominator in \eqref{density} is $\sim X/\zeta(2)$
for large $X$). In fact, it is known that


\smallskip
\noindent
\begin{tabular}{lllllll}
&$r_{=0}(X)$&$\ge$&$X/\log X$&&(Ono-Skinner \cite{OS})\cr
&$r_{=1}(X)$&$\ge$&$ X^{1-\epsilon}$&&(Pomykala-Perelli \cite{PP})\cr
&$r_{\ge 2}(X)$&$\ge$&$C_E X^{1/7}/\log^2X$&&(Stewart-Top \cite{StT}).\cr
\end{tabular}
\smallskip

\noindent
For some specific elliptic curves it is known that
$r_{=0}(X)\sim C X$ and $r_{=1}(X)\sim C' X$, but to get
$C$ or $C'$ to be $\frac 12$ seems to be extremely hard.

\subsection{No Goldfeld over number fields}

Over number fields Goldfeld's Conjecture has to be formulated differently,
because the `$w=+1$ in 50\% cases'-formula \eqref{density}
may not hold. The simplest counterexample is CM curves:

\begin{example}
Let $K=\Q(i)$ and $E/K: y^2=x^3+x$. This is a curve with complex
multiplication,
$$
  \End_K E\iso\Z[i], \qquad [i](x,y)=(-x,iy).
$$
The set of rational points $E(K)$ is naturally a $\Z[i]$-module, and so is
$E(F)$ for any extension $F/K$.
Because $E(F)\tensor_\Z\Q$ is a $\Q(i)$-vector space, it is
even-dimensional over $\Q$, so $\rk E/F$ is even for every $F\supset K$.
Hence
$$
  \rk E_d/K=\rk E/K(\sqrt d) - \rk E/K \equiv 0\mod 2,
    \qquad \text{all }d\in K^\times/K^{\times2},
$$
in other words every quadratic twist of $E/K$ has even rank.
It is also not hard to prove that $w(E_d/K)=1$ for all $d$, as expected.
\end{example}

The same applies to any CM curve $E/K$ with endomorphisms defined over $K$.
Such a $K$ is automatically totally complex (it contains $\End K$ as
a subring),
the representation $G_K\to\Aut T_l E$ has abelian image,
and $E$ acquires everywhere good reduction after an abelian extension of $K$.
Interestingly, these are precisely the local conditions on an
elliptic curve to guarantee that all of its quadratic twists have the
same root number:

\begin{theorem}[\cite{Evilquad}]
Let $E/K$ be an elliptic curve. Then $w(E_d/K)=w(E/K)$ for all
$d\in K^\times$ if and only if
\begin{itemize}
\item[(a)] $K$ is totally complex, and
\item[(b)] For all primes $\p$ of $K$ the curve $E/K_v$ acquires good
           reduction after an abelian extension of $K_v$.
\end{itemize}
\end{theorem}

For semistable curves the second condition simply says that $E$
must have everywhere good reduction. It is not hard to construct
explicit examples:

\begin{example}
The elliptic curve $E/\Q: y^2=x^3+\frac 54x^2-2x-7$ (121C1) has
minimal discriminant $11^4$ and acquires everywhere good reduction over, e.g.,
$\Q(\sqrt[3]{11})$. If we take for instance $K=\Q(\zeta_3,\sqrt[3]{11})$,
it is totally complex, $E/K$ has everywhere good reduction, so
$$
  w(E_d/K)=w(E/K)=(-1)^{\#\{v|\infty\}}=-1, \qquad \text{for all $d\in K^\times$}.
$$
The parity conjecture implies that every quadratic twist of $E/K$ has
positive rank, or, equivalently, that the rank of $E$ must grow in
{\em every\/} quadratic extension of $K$.

\begin{conjecture}
The curve 121C1 over $K=\Q(\zeta_3,\sqrt[3]{11})$ and all of its quadratic
twists over $K$ have positive rank.
\end{conjecture}

Here is a very elementary way to phrase this:

\begin{conjecture}
Over $K=\Q(\zeta_3,\sqrt[3]{11})$ the polynomial
$x^3+\frac 54x^2-2x-7\in K[x]$ takes every value in $K^\times/K^{\times2}$.
\end{conjecture}

\end{example}

\subsection{No local expression for the rank}

The reader may have noticed that the above examples rely not so much
on the precise formulae for the root numbers but mostly just on
their existence. In other words, they explore the fact that (conjecturally)
the parity of the Mordell-Weil rank is a `sum of local invariants':

\begin{definition}
Say that the Mordell-Weil rank
(resp. Mordell-Weil rank modulo $n$) is a {\em sum of local invariants\/}
if there is a $\Z$-valued function $(k,E)\mapsto\Lambda(E/k)$
of elliptic curves over local fields%
\footnote{Meaning that if $k\iso k'$ and
  $E/k$ and $E'/k'$ are isomorphic elliptic curves (identifying $k$ with $k'$),
  then $\Lambda(E/k)=\Lambda(E'/k')$.},
such that for any elliptic curve $E$ over any number field $K$,
$$
  \rk E/K = \sum\limits_v \lambda(E/K_v) \qquad
  (\text{resp. } \rk E/K \equiv \sum\limits_v \lambda(E/K_v)\mod n),
$$
the sum taken over all places of $K$ (and implicitly finite).
\end{definition}

The parity conjecture implies that the Mordell-Weil rank
modulo 2 is a sum of local invariants, namely those defined by the local
root numbers,
$$
  (-1)^{\lambda(E/K_v)}:=w(E/K_v).
$$
One might ask whether there is a local expression like this
for the rank modulo 3 or modulo 4, or even for the rank itself.
The answer is `no':

\begin{theorem}[\cite{Antifor}]
\label{thmmw}
The Mordell-Weil rank is not a sum of local invariants.
In fact, a stronger statement holds: for $n\in\{3,4,5\}$
the Mordell-Weil rank modulo $n$ is not a sum of local invariants.
\end{theorem}

\begin{proof}
Take $E/\Q: y^2=x(x+2)(x-3)$, which is 480a1 in Cremona's notation.
Writing $\zeta_p$ for a primitive $p$th root of unity, let
$$
  F_n = \left\{\begin{array}{ll}
    \text{the degree 9 subfield of }\Q(\zeta_{13},\zeta_{103})\quad & \text{if } n=3, \cr
    \text{the degree 25 subfield of }\Q(\zeta_{11},\zeta_{241}) & \text{if }n=5, \cr
    \Q(\sqrt{-1},\sqrt{41},\sqrt{73}) & \text{if }n=4.\cr
  \end{array}
  \right.
$$
Because 13 and 103 are cubes modulo one another, and all other primes are
unramified in $F_3$, every place of $\Q$ splits into 3 or 9 in $F_3$.
Similarly in $F_4$ (resp. $F_5$) every place of $\Q$ splits into a multiple
of 4 (resp. 5) places.
Hence, if the Mordell-Weil rank modulo $n$ were a sum of local
invariants, it would be $0\in\Z/n\Z$ for $E/F_n$.

However, 2-descent shows that $\rk E/F_3=\rk E/F_5=1$ and $\rk E/F_4=6$
(e.g. using Magma over all minimal non-trivial subfields of $F_n$).
\end{proof}

\begin{remark}
\label{lfun}
It is interesting to note that
the $L$-series of the curve $E\!=\,$480a1 used in the proof
over $F=F_4=\Q(\sqrt{-1},\sqrt{41},\sqrt{73})$
is formally a 4th power, in the sense that each Euler factor is:
$$
L(E/F,s) = \textstyle
       1
       \!\cdot\! \bigl( \frac {1}{1-3^{-2s}} \bigr)^4
       \bigl( \frac {1}{1-5^{-2s}} \bigr)^4
       \bigl( \frac {1}{1 + 14\cdot 7^{-2s} + 7^{2-4s}} \bigr)^4
       \bigl( \frac {1}{1 + 6\cdot 11^{-2s} + 11^{2-4s}} \bigr)^4\!...
$$
However, it is not a 4th power of an entire function, as it vanishes
to order~6 at $s=1$.
(Actually,
it is not even a square of an entire function:
a computation shows it has a simple zero at $1+2.1565479...\,i$.)

In fact,
by construction of $F$, for any $E/\Q$ the $L$-series $L(E/F,s)$
is formally a 4th power and vanishes to even order at $s=1$ by the
functional equation. Its square root has analytic
continuation to a domain including $\Re s>\frac 32$,
$\Re s<\frac 12$ and the real axis, and satisfies a functional equation
$s\leftrightarrow 2-s$, but it is not clear whether it has any arithmetic
meaning.
\end{remark}


\end{document}